\documentclass[12pt]{amsart}

\calclayout

\usepackage{xcolor}

\usepackage[lite]{amsrefs}

\usepackage{amssymb}

\usepackage[all,cmtip]{xy}

\usepackage{mathrsfs}

\usepackage{tabularx}
\usepackage{booktabs}
\usepackage[labelfont=bf,format=plain,justification=raggedright,singlelinecheck=false]{caption}

\newcommand{\cc}{\mathfrak{c}}
\newcommand{\al}{\alpha}

\newcommand{\PP}{\mathcal{P}}
\newcommand{\QQ}{\mathcal{Q}}
\newcommand{\F}{\mathbb{F}}
\newcommand{\Z}{\mathbb{Z}}
\newcommand{\X}{\mathcal{X}}

\newcommand{\PGU}{\mathrm{PGU}}

\numberwithin{equation}{section}

\theoremstyle{plain}
\newtheorem{thm}[equation]{Theorem}
\newtheorem{cor}[equation]{Corollary}
\newtheorem{lem}[equation]{Lemma}
\newtheorem{prop}[equation]{Proposition}

\theoremstyle{definition}
\newtheorem{defn}[equation]{Definition}

\theoremstyle{remark}
\newtheorem{rem}[equation]{Remark}
\newtheorem{ex}[equation]{Example}

\usepackage{bookmark}

\usepackage{hyperref}

\begin{document}

\title[On a maximal curve with the third largest genus]{Weierstrass semigroups and automorphism group of a maximal curve with the third largest genus}

\author{Peter Beelen} \address{Department of Applied Mathematics and Computer Science, Technical University of Denmark, Kongens Lyngby 2800, Denmark} \email{pabe@dtu.dk} \thanks{}
\author{Maria Montanucci} \address{Department of Applied Mathematics and Computer Science, Technical University of Denmark, Kongens Lyngby 2800, Denmark} \email{marimo@dtu.dk} \thanks{}
\author{Lara Vicino} \address{Department of Applied Mathematics and Computer Science, Technical University of Denmark, Kongens Lyngby 2800, Denmark}  \email{lavi@dtu.dk} \thanks{}

\address{}

\date{}

 \begin{abstract}
In this article we explicitly determine the Weierstrass semigroup at any point and the full automorphism group of a known $\mathbb{F}_{q^2}$-maximal curve $\mathcal{X}_3$ having the third largest genus. This curve arises as a Galois subcover of the Hermitian curve, and its uniqueness (with respect to the value of its genus) is a well-known open problem. Knowing the Weierstrass semigroups may provide a key towards solving this problem. Surprisingly enough $\mathcal{X}_3$ has many different types of Weierstrass semigroups and the set of its Weierstrass points is much richer than the set of $\mathbb{F}_{q^2}$-rational points, as instead happens for all the known maximal curves where the Weierstrass points are known.
We show that a similar exceptional behaviour does not occur in terms of automorphisms, that is, $\mathrm{Aut}(\mathcal{X}_3)$ is exactly the automorphism group inherited from the Hermitian curve, apart from small values of $q$.
\end{abstract}

\maketitle

\vspace{0.5cm}\noindent {\em Keywords}: Maximal curve, Weierstrass semigroup, Weierstrass points

\vspace{0.2cm}\noindent{\em MSC}: Primary: 11G20. Secondary: 11R58, 14H05, 14H55

\vspace{0.2cm}\noindent

\section{Introduction}

An $\mathbb{F}_{q^2}$-maximal curve $\mathcal{X}$ of genus $g$ is a projective, geometrically irreducible, non-singular algebraic curve defined over $\mathbb{F}_{q^2}$ such that the number $|\mathcal{X}(\mathbb{F}_{q^2})|$ of its $\mathbb{F}_{q^2}$-rational points attains the Hasse-Weil upper bound, namely
$$|\mathcal{X}(\mathbb{F}_{q^2})| = q^2+1+2gq.$$
$\mathbb{F}_{q^2}$-maximal curves and especially those with large genus have been intensively investigated during the last decades also in connection with coding theory and cryptography based on Goppa’s
method, see e.g. \cites{Goppa,TV}.

It is well known that the genus $g$ of an $\mathbb{F}_{q^2}$-maximal curve satisfies $g  \leq q(q - 1)/2=:g_1$, see \cite{Y}, and that $g=g_1$ if and only if $\mathcal{X}$ is $\mathbb{F}_{q^2}$-isomorphic to the Hermitian curve
$$\mathcal{H}: x^q+x=y^{q+1},$$
see \cite{RS}.  In \cite{FT} it is proven that either $g \leq \lfloor (q - 1)^2/4 \rfloor=:g_2$, or $g = g_1$. For $q$ odd, $g = g_2$ occurs if and only if $\mathcal{X}$ is $\mathbb{F}_{q^2}$-isomorphic to the
non-singular model of the plane curve of equation
$$\mathcal{X}_2: x^q + x = y^{\frac{q+1}{2}},$$

see \cite{FGT}*{Theorem 3.1}. For $q$ even, a similar (but weaker) result is obtained in \cite{AT}. Indeed the uniqueness of an $\mathbb{F}_{q^2}$-maximal curve of genus $g=g_2$ is ensured under the extra condition that the curve has a particular Weierstrass point. If the extra condition is met, then it is proven in \cite{AT} that $g = q(q - 2)/4$ if and only if $\mathcal{X}$ is $\mathbb{F}_{q^2}$-isomorphic to the non-singular model of the plane curve of equation
$$ \mathcal{Y}_2: x^{\frac{q}{2}} + \ldots + x^2 + x = y^{q+1}.$$
The value of the third largest genus an $\mathbb{F}_{q^2}$-maximal curve has been computed in \cite{KT}, where it is proven that either $g \leq \lfloor(q^2 - q + 4)/6\rfloor=:g_3$, or $g=g_2$ or $g=g_1$. In \cite{KT}*{Remark 3.4} it is shown that $g \leq g_3$ is the best bound possible, as there exist $\mathbb{F}_{q^2}$-maximal curves of genus $g_3$, namely
\begin{itemize}
\item $\mathcal{X}_3: x^{(q+1)/3} + x^{2(q+1)/3} + y^{q+1} = 0$, if $q \equiv 2 \pmod 3$;

\item $\mathcal{Y}_3: y^q - yx^{2(q-1)/3} + x^{(q-1)/3} = 0$, if $q \equiv 1 \pmod 3$; and

\item $\mathcal{Z}_3: y^q + y + (\sum_{i=1}^{t} x^{q/p^i})^2 = 0$, if $q = 3^t$.
\end{itemize}
More precisely, $\mathcal{X}_3, \mathcal{Y}_3$ and $\mathcal{Z}_3$ are the non-singular models of the plane curves given by these equations. All the examples above arise as degree $3$ Galois subcovers of the Hermitian curve $\mathcal{H}_q$. Understanding whether or not these curves are the only $\mathbb{F}_{q^2}$-maximal curves of genus $g_3$ is a well-known open problem.

In the proofs of the uniqueness (up to isomorphism) of $\mathbb{F}_{q^2}$-maximal curves of genus $g_1$ and $g_2$ from \cite{RS} and \cite{FT} the so-called Weierstrass semigroups and Weierstrass points played a crucial role.
These objects occur also naturally in the study of algebraic-geometry (AG) codes \cite{TV}, as the main ingredient to construct one-point AG codes.

Given a point $P$ on an algebraic curve $\mathcal{X}$, the Weierstrass semigroup $H(P)$ is defined as the set of natural numbers $n$ for which there exists a function $f$ on $\mathcal{X}$ having pole divisor $(f)_\infty = nP$.
According to the Weierstrass gap Theorem, see \cite{Sti}*{Theorem 1.6.8}, the set $G(P) := \mathbb{N} \setminus H(P)$ contains exactly $g$ elements called gaps. The structure of $H(P)$ in general varies as $P \in \mathcal{X}$ varies. However, it is known that generically the semigroup $H(P)$ is the same, but that there can exist finitely many points of $\mathcal{X}$, called Weierstrass points, with a different set of gaps.

The intrinsic theoretical interest on these objects arises from St\"ohr-Voloch theory \cite{SV}, where (together with the so-called order sequence) Weierstrass semigroups and points are used to obtain characterizing properties of the curve. Apart from the two characterizations for $g_1$ and $g_2$ mentioned above, important characterization results using Weierstrass points, semigroups and automorphism groups can be found in \cite{FGT}, \cite{FT1} (for the Suzuki curve) and \cite{TT} (for the Ree curve).

In order to provide similar tools for maximal curves of third largest genus, it is natural to wonder whether Weierstrass semigroups and points can be completely determined for maximal curves of genus $g_3$.

In this paper we compute the set the Weierstrass semigroup at every point of the curve $\mathcal{X}_3$ having third largest genus $g_3$ for $q \equiv 2 \pmod 3$, as well as its set of Weierstrass points and its full automorphism group.

In all known examples of maximal curves with large enough genus the set of Weierstrass points coincides with that of rational points, see e.g. \cites{BMZ,BM,BLM,GV}. To understand this behaviour previous investigations found sufficient conditions for a maximal curve to satisfy this property, see \cite{GT1}.

In this paper we show on the contrary that the curve $\mathcal{X}_3$ has a quite large set of non-rational Weierstrass points and many different type of Weierstrass semigroups, providing the first known example with these features. The full automorphism group of $\mathcal{X}_3$ is also computed, as an application of the results mentioned above.

The paper is organized as follows. Section 2 provides the necessary preliminary results on algebraic curves, Weierstrass semigroups and the curve $\mathcal{X}_3$. Section 3 presents two family of special functions in $\mathbb{F}_{q^2}(\mathcal{X}_3)$ and their relations. These functions represent the main ingredient used to compute the Weierstrass semigroups $H(P)$ apart from a few special cases of $P$. Sections 4 and 5 are devoted to the proofs of the main theorems of the paper, namely the description of the Weierstrass semigroup at $\mathbb{F}_{q^2}$-rational and not $\mathbb{F}_{q^2}$-rational points of $\mathcal{X}_3$, respectively. In Section 6 the full automorphism group of $\mathcal{X}_3$ is computed as an application of the obtained results on Weierstass semigroups of $\mathcal{X}_3$. 

\section{Preliminaries}
\label{sec:preliminaries}
In this section, we deal with the preliminary notions and results that will be needed throughout the paper. In the first subsection, we recall the definition of the curve $\X_3$ and we focus on some particular rational functions defined on it, computing their principal divisors. In the second subsection, we collect some preliminaries on regular differentials. In particular, we compute a canonical divisor and prove Corollary \ref{cor:maintoolgaps}, that we will need in Sections \ref{sec:rational} and \ref{sec:nonrational}.

\subsection{The curve \texorpdfstring{$\mathcal{X}_3$}{X3}}

Let $q$ be a prime power such that $q \equiv 2 \ (\mathrm{mod} \ 3)$ and define $m:=\frac{q+1}{3}$. Let $\F_{q^2}$ be the finite field with $q^2$ elements and denote by $p$ the characteristic of $\F_{q^2}$. As before, let $\mathcal{X}_3$ be the non-singular model of the plane curve with affine equation
\begin{equation}\label{eq:planecurvemodel}
y^{q+1} + x^{2m} + x^m = 0.
\end{equation}
The function field $\F_{q^2}(\X_3)$ can be described as $\F_{q^2}(x,y)$, with $y^{q+1} + x^{2m} + x^m = 0$, and it is easy to see that $\F_{q^2}(\X_3)/\F_{q^2}(x)$ is a Kummer extension of degree $q+1$.

\begin{rem}
The point $(0,0)$ is a singular point of the curve defined by equation \eqref{eq:planecurvemodel}. Considering what we have just observed on the desingularization $\X_3$ and its function field, it is then easy to see that there are exactly $m$ distinct places centered on $(0,0)$ in $\F_{q^2}(\X_3)$, and we denote as $P_{0}^1,\ldots ,P_{0}^m$ the distinct points of $\X_3$ that are the centers of such places.
The point at infinity of the plane curve is singular as well. Also for this point there are exactly $m$ distinct places centered on it in $\F_{q^2}(\X_3)$.
We denote by $P_{\infty}^1,\ldots ,P_{\infty}^m$ the distinct points of $\X_3$ that are the centers of these $m$ places.
\end{rem}

Consider the $\F_{q^2}$-model of the Hermitian curve $\mathcal{H}$ given by
\begin{equation*}
    \mathcal{H}: u^{q+1} + v^{q+1} + 1 =0
\end{equation*}
and let $\F_{q^2}(\mathcal{H})$ be the function field of $\mathcal{H}$, that can be described as $\F_{q^2}(u,v)$, with $u^{q+1} + v^{q+1} + 1 =0$. The Hermitian curve is a nonsingular plane curve of genus $g(\mathcal{H})=q(q-1)/2$. The curve $\mathcal{X}_3$ is $\F_{q^2}$-maximal of genus $g(\X_3):=\frac{q^2-q+4}{6}$ and it is $\F_{q^2}$-covered by the Hermitian curve $\mathcal{H}$ via a morphism $\varphi$ of degree $3$. More precisely, the pull-back map
\begin{equation*}
    \varphi^*: \F_{q^2}(\X_3) \longrightarrow \F_{q^2}(\mathcal{H})
\end{equation*}
defines a Galois extension $\F_{q^2}(\mathcal{H})/\F_{q^2}(\X_3)$ of degree $3$, with $x:=u^3$ and $y:=uv$. In particular, the Galois group of the extension is generated by the automorphism
\begin{equation}
    \label{eq:tau}
    \tau: (u,v) \longmapsto (\zeta_3 u , \zeta_3^{2} v),
\end{equation}
where $\zeta_3$ is a primitive cube root of unity in $\F_{q^2}$.

\begin{rem}
\label{rem:unramified}
The extension $\F_{q^2}(\mathcal{H})/\F_{q^2}(\X_3)$ is unramified: indeed, by the Hurwitz genus formula, it holds
\begin{equation*}
\mathrm{deg} \ \mathrm{Diff} (\F_{q^2}(\mathcal{H})/\F_{q^2}(\X_3)) =  2g(\mathcal{H}) - 2 -3 (2g(\X_3) - 2)=0,
%
\end{equation*}
which means precisely that the extension $\F_{q^2}(\mathcal{H})/\F_{q^2}(\X_3)$ is unramified.
As a consequence of this fact, if $Q$ is a point of $\mathcal{H}$ lying over the point $P$ of $\X_3$, then for any $f\in \F_{q^2}(\X_3)$,  it holds that $v_Q(f) = v_P(f)$.
\end{rem}

For convenience, we define ${\mathcal O}_0:=\{P_{0}^1,\ldots ,P_{0}^m\}$ and ${\mathcal O}_\infty:=\{P_{\infty}^1,\ldots ,P_{\infty}^m\}$.
Throughout the paper, we denote with $P_{(a,b)}\in \X_3(\overline{\F}_{q^2})$ a point of $\X_3$ not in the set of points ${\mathcal O}_0 \cup {\mathcal O}_\infty$, with $(x,y)$-coordinates $(a,b)$. With this convention in mind, in particular, the $x$-coordinate $a$ of such a point $P_{(a,b)}$ is nonzero.
Analogously, we denote with $Q_{(A,B)}\in \mathcal{H}$ a point of $\mathcal{H}$ with $(u,v)$-coordinates $(A,B)$ and not lying over any point of $\X_3$ in the set ${\mathcal O}_0 \cup {\mathcal O}_\infty$. The points in the set ${\mathcal O}_0 \cup {\mathcal O}_\infty$ turn out to lie in the same orbit under the action of the full automorphism group of $\X_3$. In fact, the orbit turns out to be given by:
\begin{equation}
    \label{eq:O}
    \mathcal{O}:={\mathcal O}_0 \cup {\mathcal O}_\infty \cup {\mathcal O}_m,
\end{equation}
where ${\mathcal O}_m:=\{P_{(a,0)} \mid a^m + 1= 0\}.$
It will be convenient to establish some of these results already now, so that we will be able to determine the Weierstrass semigroups of points in this orbit efficiently. In Section \ref{sec:automorphisms}, we will then continue the study of automorphisms on $\X_3$.

\begin{lem} \label{lem:someauto}
The automorphism group $\mathrm{Aut}(\mathcal{X}_3)$ contains a subgroup $G$ of order $2(q+1)^2$ which is isomorphic to a semidirect product of an abelian group $A$ of order $(q+1)^2/3$ and a symmetric group of order $6$. More precisely,
$$A:=\{\theta_{\gamma,\delta}(x,y)=(\gamma x, \delta y) \mid \gamma^m=\delta^{q+1}=1\},$$
while the symmetric group of order $6$ is generated by the involution $\theta_2$ and the order $3$ automorphism $\theta_3$ given by
$$\theta_2(x,y)=\bigg( \frac{1}{x}, \frac{y}{x} \bigg) \text{ and } \theta_3(x,y)=\bigg(\frac{y^3}{x^2}, \frac{y}{x}\bigg).$$
\end{lem}

\begin{proof}
By direct computation it can be checked that $\langle A,\theta_2,\theta_3 \rangle$ is an automorphism group of $\mathcal{X}_3$, that is, all the maps presented in the lemma preserve the equation $y^{q+1}+x^m+x^{2m}=0$.
The group $T$ generated by $\theta_2$ and $\theta_3$ is isomorphic to the symmetric group of order $6$ as $\theta_2 \theta_3 \theta_2=\theta_3^2$, by direct computation. Both $\theta_2$ and $\theta_3$ normalize $A$, since a direct computation shows $\theta_2 \theta_{\gamma,\delta} \theta_2=\theta_{\gamma^{-1},\delta \gamma^{-1}}$ and $\theta_3 \theta_{\gamma,\delta} \theta_3^{-1}=\theta_{\gamma \delta^{-3},\gamma \delta^{-2}}.$ Hence $T$ normalizes $A$. Since $T$ and $A$ have trivial intersection, we hence obtain that $\langle A,T \rangle = A \rtimes T$.
\end{proof}

Next let us determine divisors of several elementary function in $\F_{q^2}(\X_3)$. We denote as $D_\infty$ the divisor
\begin{equation*}
    D_\infty:=\sum_{j=1}^{m}P_{\infty}^j.
\end{equation*}

Since $\mathcal{X}_3$ is $\mathbb{F}_{q^2}$-maximal, from the Fundamental Equation \cite{HKT}*{Page xix (ii)} it follows in particular that, for all $i=1,\ldots,m$ and $P_{(a,b)} \in \mathcal{X}_3$, there exists a function $f_{P_{(a,b)},i} \in \F_{q^2}(\X_3)$ such that
\begin{equation} \label{fundeq}
(f_{P_{(a,b)},i})=qP_{(a,b)}+\Phi(P_{(a,b)})-(q+1)P_\infty^i.
\end{equation}
Here $\Phi$ denotes the $\mathbb{F}_{q^2}$-Frobenius map. For a point $P_{(a,b)}$ of $\X_3$, we define the function
\begin{equation*}
    x_a:=\frac{x-a}{a},
\end{equation*}
which, as we will see later, turns out to be a local parameter for $P_{(a,b)}$. Further, let $t_{P(a,b)}$ be the following function in $\F_{q^2}(\X_3)$
\begin{equation}
\label{eq:tangent}
    t_{P_{(a,b)}}:=ma^{m-1}(2a^m+1)(x-a) + b^q(y-b),
\end{equation}
and let $Q_{(A,B)}$ be a point of $\mathcal{H}$ lying over $P_{(a,b)}$.
Note that $t_{P_{(a,b)}}$ is the function associated to the tangent line at $(a,b)$ of the plane curve defined by equation \eqref{eq:planecurvemodel}.

With $\mathcal{O}$ defined as in equation \eqref{eq:O}, for $P_{(a,b)} \in \X_3(\overline{\F}_{q^2}) \setminus \mathcal{O}$, we define
\begin{equation}
\label{eq:alpha}
    \al(P_{(a,b)}):=\frac{a^m}{1+a^m} = \frac{A^{q+1}}{1+A^{q+1}}.
\end{equation}
As $1 - \al(P_{(a,b)})= \frac{1}{1+a^m}$, in particular $1 - \al(P_{(a,b)}) \neq 0$ and we can define the following nonzero function in $\F_{q^2}(\X_3)$, that will be useful later:
\begin{equation}
\label{eq:f0}
    f_0:=\frac{3(1 - \al(P_{(a,b)}))}{A^{q+1}}t_{P_{(a,b)}} = (1 - \al(P_{(a,b)}))\left(\frac{\left(2A^{q+1}+1\right)}{A^3}(x-a) + \frac{3B^q}{A}(y-b) \right),
\end{equation}
where $A^3 = a$ and $AB = b$.
Given a point $P_{(a,b)}\in \X_3(\overline{\F}_{q^2})\setminus \mathcal{O}$ and $Q_{(A,B)}$ a point of $\mathcal{H}(\overline{\F}_{q^2})$ lying over $P_{(a,b)}$, the following proposition describes the local power series expansion of the functions $x_a$ and $f_0$ at $Q_{(A,B)}$, with respect to the local parameter $T:=\frac{u-A}{A}$. In this proposition as well as in the remainder of this article, whenever we write $f=g+O(T^n)$, we mean that $v_{P_{(a,b)}}(f-g) \ge n$.

\begin{prop}
\label{prop:power-series}
Let $P_{(a,b)}\in \X_3(\overline{\F}_{q^2})\setminus \mathcal{O}$ and $Q_{(A,B)}$ a point of $\mathcal{H}$ lying over $P_{(a,b)}$. Consider the functions $x_a$, $f_0$ and $T:=\frac{u-A}{A}$, which is a local parameter at $Q_{(A,B)}$. Then, the local power series expansions of $x_a$ and $f_0$ at $Q_{(A,B)}$ with respect to $T$ are
\begin{equation}
\label{eq:f0:exp}
    \begin{split}
        x_a &= 3T + 3T^2 + T^3,\\
        f_0 &= 3T^2 + (\alpha(P_{(a,b)})+1)T^3 + O(T^{q}),
    \end{split}
\end{equation}
where $\alpha(P_{(a,b)})$ is as defined in equation \eqref{eq:alpha}.
\end{prop}

\begin{proof}
For convenience, we will simply write $\al$ instead of $\al(P_{(a,b)})$ in this proof.
We start by computing the local power series expansions of the functions $x-a$ and $y-b$ with respect to the local parameter $T:=(u-A)/A$ at $Q_{(A,B)}$. We have:
\begin{equation*}
x_a=\frac{x-a}{a} = \frac{x-A^3}{A^3} =  \frac{u^3 - A^3}{A^3} = \frac{(u-A)^3+3A(u-A)^2+3A^2(u-A)}{A^3}=3T + 3T^2 + T^3
%
\end{equation*}
and
\begin{equation}\label{eq:yminb}
\begin{split}
    y-b = uv - AB &= (u-A)(v-B) + B(u-A) + A(v-B) -AB + AB\\
    &=A(v-B)(T+1) + AB T.
\end{split}
\end{equation}
Moreover, from $v^{q+1} + u^{q+1} +1 = 0$, we obtain
\[        (u-A)^{q+1} -A^{q+1} +A^q u + Au^q + (v-B)^{q+1} - B^{q+1} +B^q v +Bv^q +1 = 0\]
or equivalently
\begin{equation*}
\label{eq:vB1}
       A^{q+1} T^{q+1} +(v-B)^{q+1} + A^{q+1} T^q +B(v-B)^q + A^{q+1} T + B^q(v-B)=0
\end{equation*}
%
which gives $v-B=-\frac{A^{q+1}}{B^q}T+O(T^q)$. Combining this with equation \eqref{eq:yminb}, we obtain
\begin{equation*}
    \begin{split}
        y - b &= A(v-B)(T+1) + ABT=-A\frac{A^{q+1}}{B^q}T(T+1)+ABT+O(T^{q})\\
        &= A\left(B - \frac{A^{q+1}}{B^q}\right) T - \frac{A^{q+2}}{B^q}T^2 + O(T^{q}).
    \end{split}
\end{equation*}
We can now compute also the local power series expansion of the function $f_0$ at $Q_{(A,B)}$ with respect to the local parameter $T$. Using equation \eqref{eq:f0} and the previously computed expansions of $x_a$ and $y-b$, we find
\begin{equation*}
    \begin{split}
        f_0  &= (1 - \al)(3\left(A^{q+1} + 1\right)T^2 + \left(2A^{q+1} + 1\right)T^3) + O(T^{q})\\
             &= 3T^2 + (\al+1)T^3 + O(T^{q}),
    \end{split}
\end{equation*}
where in the final equality we used that $\al=A^{q+1}/(1+A^{q+1}).$
\end{proof}

\begin{prop}
\label{prop:divisors-on-X3}
In the above notations, the principal divisors of the functions $x, x_a, y, y-b$ and $f_0$ in $\F_{q^2}(\X_3)$ are:
\begin{equation}
\label{eq:div:xa}
    (x_a)=\begin{cases}
         & P_{(a,b)} + \sum_{\xi^{q+1}=1, \ \xi \neq 1} P_{(a,\xi b)} - 3D_\infty \quad \mbox{if} \quad a^m\neq -1,\\
       \\[1pt]
        & (q+1)P_{(a,0)} - 3D_\infty \quad \mbox{if} \quad a^m = -1,\\
    \end{cases}
\end{equation}
and
\begin{equation}
\label{eq:divs}
\begin{split}
    (x) &=  3\sum_{i=j}^{m}P_{0}^j - 3D_\infty,\\
    (y) &= \sum_{j=1}^{m}P_{0}^j + \sum_{a^m+1=0}P_{(a,0)} - 2D_\infty,\\
    (y-b) &= P_{(a,b)}+E_{b} - 2D_\infty,\\
\end{split}
\end{equation}
where $E_{b}\in \mathrm{Div}(\X_3)$ is an effective divisor of degree $2m-1$. Moreover, if $P_{(a,b)}\in \X_3(\overline{\F}_{q^2})\setminus \mathcal{O}$ and $\alpha(P_{(a,b)})\neq -1$, then $P_{(a,b)} \not \in \mathrm{supp}(E_b)$. Further, for $P_{(a,b)}\in \X_3(\overline{\F}_{q^2})\setminus \mathcal{O}$, let $f_0$ be the function defined in equation \eqref{eq:f0}. Then
\begin{equation}
\label{eq:div:f0}
(f_0) = 2P_{(a,b)} + E_0 - 3D_\infty,
\end{equation}
where $E_0\in \mathrm{Div}(\X_3)$ is an effective divisor such that $P_{(a,b)}\not \in \mathrm{supp}(E_0)$.
\end{prop}

\begin{proof}
To find the divisors of $x_a$ and $x$, observe that $\F_{q^2}(\X_3)/\F_{q^2}(x)$ is a Kummer extension of degree $q+1$. Then, it is sufficient to note that the zeros of $x^m + 1$ are totally ramified in this extension, while, the zero and the pole of $x$ have ramification index $3$. No further ramification occurs as $y^{q+1} = -x^m(x^m + 1)$. This equation also gives the divisor of $y$. It is not clear that the divisor of $y-b$ is of the form as stated in the proposition, but it might happen that $P_{(a,b)} \in \mathrm{supp}(E_b)$. In this case, the polynomial $f(x):=x^{2m}+x^m+b^{q+1}$ would have $a$ as a multiple root. Since $3f'(x)=x^{m-1}(2x^{m}+1)$ and $P_{(a,b)} \not\in \mathcal{O}$, this can only happen if $2a^m+1=0$. Using that $\alpha(P_{(a,b)})+1=(2a^m+1)/(a^m+1)$, the result on the divisor of $y-b$ follows.

Finally, from equation \eqref{eq:f0:exp}, we know that $v_{P_{(a,b)}}(f_0)= 2$ and, as $f_0$ is a linear combination of $x_a$ and $y-b$, by the triangle inequality we also know that $v_{P_{\infty}^j}(f_0)= -3$ and that $f_0$ has no poles outside the $P_{\infty}^j$, $1\leq j\leq m$. Hence, equation \eqref{eq:div:f0} follows.
\end{proof}

\begin{cor} \label{cor:orbitO}
Let $\mathcal{O}$ be the set defined in equation \eqref{eq:O}. Then $|\mathcal{O}|=q+1$ and $\mathcal{O}$ is an orbit of the automorphism group $G$ defined in Lemma \ref{lem:someauto}, in its natural action on the points of $\X_3$.
\end{cor}

\begin{proof}
We observe that the Galois group of the extension $\F_{q^2}(\X_3)/\F_{q^2}(x)$, that is the cyclic group generated by $(x,y) \longmapsto (x,\delta y)$, where $\delta$ is a primitive $(q+1)$-th root of unity, fixes the set ${\mathcal O}_m$ point-wise, while it acts transitively on the sets ${\mathcal O}_0$ and ${\mathcal O}_\infty$.
The group $A$ as defined in Lemma \ref{lem:someauto}, acts transitively on the set ${\mathcal O}_m$ because it maps $x$ to $\gamma x$, where $\gamma^m=1$. The automorphism $\theta_2$ maps $x$ to $1/x$ and hence from equation \eqref{eq:divs} merges the two orbits ${\mathcal O}_0$ and ${\mathcal O}_\infty$ under the action of $G$. The automorphism $\theta_3$ instead acts as a cycle of order $3$ on ${\mathcal O}_0$, ${\mathcal O}_\infty$ and ${\mathcal O}_m$. This can be seen from equation \eqref{eq:divs} and the fact that $\theta_3$ maps $x$ to $y^3/x^2$. As a result, all the three considered sets are merged into one orbit under the action of $G$.
\end{proof}

\begin{rem}
Let $\bar{a}\in \F_{q^2}$ be such that $\bar{a}^m=-1$. Then, the Fundamental Equation \cite{HKT}*{Page xix (ii)} ensures that, for any point $P_{(a,b)}\in \X_3(\overline{\F}_{q^2})\setminus \X_3(\F_{q^2})$, there exists a function $\phi_{P_{(a,b)}}\in \F_{q^2}(\X_3)$ such that
\begin{equation*}
    (\phi_{P_{(a,b)}}) = qP_{(a,b)} + \Phi(P_{(a,b)}) - (q+1)P_{(\bar{a},0)}.
\end{equation*}
Hence, we can consider the following function in $\F_{q^2}(\X_3)$, that will be useful later:
\begin{equation}
\label{eq:F}
    F_{P_{(a,b)}}:=\phi_{P_{(a,b)}} \cdot x_{\bar{a}}.
\end{equation}
By Proposition \ref{prop:divisors-on-X3}, the principal divisor of $F_{P_{(a,b)}}$ is
\begin{equation*}
    (F_{P_{(a,b)}}) = qP_{(a,b)} + \Phi(P_{(a,b)}) - 3\sum_{j=1}^{m}P_{\infty}^j.
\end{equation*}
\end{rem}

\subsection{Regular differentials and gaps}

In this subsection, we recall a fundamental result relating regular differentials on a curve and the gaps at a point of the curve. More specifically, in Lemma \ref{lem:canonicaldiv} we compute a particular canonical divisor on $\X_3$ and, in Corollary \ref{cor:maintoolgaps}, we show how to use this for determining gaps at certain points of the curve.
In particular, Corollary \ref{cor:maintoolgaps} will be crucial for the results in Section \ref{sec:nonrational}.

\begin{prop} \label{prop:villa}\cite{VS}*{Corollary 14.2.5} Let $\mathcal{X}$ be an algebraic curve of genus $g$ defined
over a field $\mathbb K$. Let $P$ be a point of $\mathcal{X}$ and $w$ be a regular differential on $\mathcal{X}$. Then $v_P (w)+1$
is a gap at $P$.
\end{prop}

\begin{lem} \label{lem:canonicaldiv}
The divisor $(q-2)D_\infty$ is canonical. More precisely
$$\left(\frac{ydx}{x(x^m+1)}\right)=(q-2)D_\infty.$$
\end{lem}
\begin{proof}
The result follows directly from the fact that $\F_{q^2}(\X_3)/ \mathbb{F}_{q^2}(x)$ is a Kummer extension of degree $q+1$. Indeed, as observed in the proof of Proposition \ref{prop:divisors-on-X3}, we have that the zeros of $x^m + 1$ in $\mathbb{F}_{q^2}(x)$ are totally ramified in the extension, the zeros and the poles of $x$ have ramification exponent $3$ and the points that are not zeros of $x^m(x^m + 1)$ split completely.
Hence, we have
$$
(dx)=2\sum_{j=1}^m P_0^j + q\sum_{a^m+1=0}P_{(a,0)} - 4 D_\infty$$
and the claim now follows from direct computation using Proposition \ref{prop:divisors-on-X3}.
\end{proof}

\begin{cor} \label{cor:maintoolgaps}Let $P$ be a point of $\mathcal{X}_3$ not in ${\mathcal O}_\infty$. Then for any $h  \in L((q-2)D_\infty)$, the integer $v_P (h) + 1$ is a gap of Weierstrass semigroup at $P$.
\end{cor}
\begin{proof}
From Proposition \ref{prop:villa} and Lemma \ref{lem:canonicaldiv} it is enough to consider the regular differential
$$w:=\frac{hydx}{x(x^m+1)}.$$
Since $v_P(w)=v_P(h)$, the corollary follows.
\end{proof}

\section{Two families of functions in $\mathbb{F}_{q^2}(\X_3)$}
\label{sec:twofamilies}

The aim of this section is to prove Theorem \ref{thm:fi} and Theorem \ref{thm:gi}, that introduce two families of functions in $\F_{q^2}(\X_3)$ with prescribed vanishing orders in certain points of the curve. These functions will be crucial for the computation of the Weierstrass semigroups at the points of $\X_3 \setminus \mathcal{O}$.

We start by giving the following definition, introducing some functions that will be practical to use in the proofs of Theorem \ref{thm:fi} and Theorem \ref{thm:gi}.

\begin{defn}\label{def:PQ}
Let $i \in \Z$. Further, let $\F$ be a field of characteristic different from three and assume that it contains a primitive cube root of unity, which we will denote by $\zeta_3$. Then we define the following rational functions in $\F(s)$:
\begin{equation*}
    \PP_i(s):=\frac{(s + \zeta_3)^{3i} - (s + \zeta_{3}^2)^{3i}}{3(\zeta_3 - \zeta_{3}^2)s(s-1)}
\end{equation*}
and
\begin{equation*}
    \QQ_{i}(s):=\frac{\left(\frac{1-\zeta_3}{3}\right)(s + \zeta_3)^{3i-1} + \left(\frac{1-\zeta_{3}^2}{3}\right)(s + \zeta_{3}^2)^{3i-1}}{s-1}.
\end{equation*}
\end{defn}
Note that it strictly speaking is not necessary to assume that the field $\F$ contains a primitive cube root of unity. If it does not, the above definition makes sense over the larger field $\F(\zeta_3)$, but actually elementary Galois theory can be used to show $\mathcal{P}_i(s)$ and $\mathcal{Q}_{i}(s)$ are in $\F(s)$.

\begin{ex}\label{ex:PQ}
Assume $\F=\mathbb{Q}$. Then $\PP_0(s)=0$, $\PP_1(s)=1$, $\PP_2(s)=2s^3-3s^2-3s+2$ and $\PP_3(s)=3s^6 - 9s^5 - 9s^4 + 33s^3 - 9s^2 - 9s + 3$. Moreover, $\QQ_1(s)=s+1$, $\QQ_2(s)=s^4 + s^3 - 9s^2 + s + 1$, $\QQ_3(s)=s^7 + s^6 - 27s^5 + 29s^4 + 29s^3 - 27s^2 + s + 1$, and $\QQ_{0}(s)=(s^2-s+1)^{-1}$.
\end{ex}

In fact, as illustrated in this example, for positive values of $i$ the rational functions $\PP_i(s)$ and $\QQ_{j}(s)$ are polynomials in $s$. We investigate this further in the following lemma.
\begin{lem}
\label{lem:PQ:degree}
Let $i \in \Z_{>0}$. Then $\PP_i(s)$ is a nonzero polynomial of degree at most $3i-3$, while $\QQ_{i}(s)$ is a nonzero polynomial of degree $3i-2$.
\end{lem}
\begin{proof}
It is easy to see that for any $i\in \Z_{> 0}$, the polynomial $\tilde{\PP}_i(s):=(s+\zeta_3)^{3i}-(s+\zeta_3^2)^{3i}$ has at most degree $3i-1$. It is not the zero polynomial, since if $s$ is substituted by $-\zeta_3$, one obtains
\begin{equation}\label{eq:Pnonzero}
\tilde{\PP}_i(-\zeta_3)=0^{3i}-(-\zeta_3+\zeta_3^2)^{3i}=(-1+\zeta_3)^{3i},
\end{equation}
which is not zero, as $\zeta_3 \neq 1.$ Here we used that the field $\F$ does not have characteristic three.
It is easy to see that $\tilde{\PP}_i(0)=0$, while
$$\tilde{\PP}_i(1)=(1+\zeta_3)^{3i}-(1+\zeta_3^2)^{3i}=(-\zeta_3^2)^{3i}-(-\zeta_3)^{3i}=(-1)^{i}-(-1)^{i}=0.$$
We may conclude that $\mathcal{P}_i(s)$ is a polynomial of degree at most $3i-3$.
Similarly, the polynomial $\tilde{\QQ}_{i}(s):=\frac{1-\zeta_3}{3}(s+\zeta_3)^{3i-1}+\frac{1-\zeta_3^2}{3}(s+\zeta_3^2)^{3i-1}$ is a polynomial of degree $3i-1$ having $1$ as a root. Hence $\QQ_{i}(s)$ is a polynomial of degree $3i-2.$
\end{proof}

It is not hard to see that the coefficient of $s^{3i-3}$ of the polynomial $\PP_i(s)$ equals $3i(\zeta_3-\zeta_3^2)$. Hence if the characteristic of the field $\F$, which already is assumed to be distinct from three, is zero or does not divide $i$, then the degree of $\PP_i(s)$ is exactly $3i-3$. Since we will work over the finite field $\F_{q^2}$, where $q \equiv 2 \pmod{3}$, it may well happen that $\deg \PP_i(s) < 3i-3.$

The following lemma gives a relation between the rational functions just introduced that will come in handy later.

\begin{lem}
\label{lem:PQ:identities}
Let $i,j,\ell \in \Z$. Then
\begin{equation}
    \label{eq:id:1}
    \PP_i(s)\PP_{\ell + j}(s)-\PP_j(s) \PP_{\ell+i}(s) =(s^2-s+1)^{3j} \PP_{i-j}(s) \PP_{\ell}(s)
\end{equation}
and
\begin{equation}
    \label{eq:id:2}
    \PP_i(s) \QQ_{\ell+j}(s) -\PP_j(s) \QQ_{\ell+i}(s) =(s^2-s+1)^{3j}\PP_{i-j}(s) \QQ_{\ell}(s).
\end{equation}
\end{lem}
\begin{proof}
For convenience, we will simply write $\PP_i$ and $\QQ_j$ instead of $\PP_i(s)$ and $\QQ_i(s)$ in this proof. We prove the second identity only, since the first identity can be proven in a very similar way with simpler looking intermediate expressions. First of all, using Definition \ref{def:PQ} and writing $S_1=s+\zeta_3$, $S_2=s+\zeta_3^2$, one obtains by direct computation
\begin{multline*}
3(\zeta_3-\zeta_3^2)s(s-1)^2\PP_i \QQ_{\ell+j} =\\
\frac{1-\zeta_3}{3}{S_1^{3i+3j+3\ell-1}}+
\frac{1-\zeta_3^2}{3}S_1^{3i}{S_2^{3j+3\ell-1}}
-\frac{1-\zeta_3}{3}{S_2^{3i}S_1^{3j+3\ell-1}}
-\frac{1-\zeta_3^2}{3}{S_2^{3i+3j+3\ell-1}}
\end{multline*}
and
\begin{multline*}
3(\zeta_3-\zeta_3^2)s(s-1)^2\PP_j \QQ_{\ell+i} =\\
\frac{1-\zeta_3}{3}{S_1^{3i+3j+3\ell-1}}+
\frac{1-\zeta_3^2}{3}S_1^{3j}{S_2^{3i+3\ell-1}}
-\frac{1-\zeta_3}{3}{S_2^{3j}S_1^{3i+3\ell-1}}
-\frac{1-\zeta_3^2}{3}{S_2^{3i+3j+3\ell-1}}.
\end{multline*}
Hence
\begin{multline*}
3(\zeta_3-\zeta_3^2)s(s-1)^2(\PP_i \QQ_{\ell+j} -\PP_j \QQ_{\ell+i} )=\\
(S_1S_2)^{3j}\left(\frac{1-\zeta_3}{3}(S_1^{3\ell+3i-3j-1}-S_2^{3i-3j}S_1^{3\ell-1})+\frac{1-\zeta_3^2}{3}(S_1^{3i-3j}S_2^{3\ell-1}-S_2^{3\ell+3i-3j-1})\right)\\
=(S_1S_2)^{3j}(S_1^{3(i-j)}-S_2^{3(i-j)})\left(\frac{1-\zeta_3}{3}S_1^{3\ell-1}+\frac{1-\zeta_3^2}{3}S_2^{3\ell-1}\right)\\
=3(\zeta_3-\zeta_3^2)s(s-1)^2 (s^2-s+1)^{3j} \PP_{i-j} \QQ_{\ell}.
\end{multline*}
For the last equality, note that $S_1S_2=s^2-s+1$.
\end{proof}

\begin{rem}
\label{rem:PQ:nocommon:roots}
For any $i \in \Z_{>0}$, the polynomials $\PP_i(s)$ and $\QQ_i(s)$ have no common roots. Indeed, this is clear for $i=1$, since $\PP_1(s)=1$. If $i \ge 2$, Lemma \ref{lem:PQ:identities} applied with $\ell=0$ and $j=i-1$, implies that $\PP_i(s) \QQ_{i-1}(s)-\PP_{i-1}(s)\QQ_i(s)=(s^2-s+1)^{3i-4}$. Here we used that $\QQ_0(s)=(s^2-s+1)^{-1}.$ Hence the only possible common roots of $\PP_i(s)$ and $\QQ_i(s)$ could be $-\zeta_3$ or $-\zeta_3^2$, the roots of $s^2-s+1$. However, equation \eqref{eq:Pnonzero} implies that $\PP_i(-\zeta_3) \neq 0$ and similarly one sees that $\PP_i(-\zeta_3^2) \neq 0$.
\end{rem}

\begin{rem}
\label{rem:PQ:existence_of_i}
Let $\F=\overline{\F}_{q^2}$ be the algebraic closure of $\F_{q^2}.$ Then for any $\alpha \in \F \setminus \{0,1,-\zeta_3,-\zeta_3^2\}$, there exists $i>0$ such that $\PP_{i+1}(\alpha)=0$. Indeed, for such $\alpha$ one has $\PP_{i+1}(\alpha)=0$ if and only if $((\al+\zeta_3)/(\al+\zeta_3^2))^{3i+3}=1$. Since any nonzero element of $\F$ has a finite multiplicative order, the existence of $i$ follows. Moreover, since $\PP_1(s)=1$, we see that $i>0$.
\end{rem}
This remark motivates the following definition:

\begin{defn}
Let $\alpha \in \overline{\F}_{q^2} \setminus \{0,1,-\zeta_3,-\zeta_3^2\}$. Then we define the $\PP$-order of $\alpha$ as the smallest positive integer $i$ such that $\PP_{i+1}(\alpha)=0$.
\end{defn}

Later we will apply the notion of a $\PP$-order in case $\alpha=\alpha(P_{(a,b)})$. The following lemma is a first source of information in this setting.

\begin{lem}\label{lem:Porder_rational}
Let $i$ be a positive integer. The number of $P_{(a,b)} \in \X_3(\overline{\F}_{q^2}) \setminus \mathcal{O}$ such that $\alpha(P_{(a,b)})$ has $\PP$-order $i$ is equal to $(q+1)^2 \varphi(i+1)$ if $\gcd(i+1,p)=1$ and $0$ otherwise. Here $\varphi(\cdot)$ denotes Euler's totient function. Moreover, $P_{(a,b)} \in  \X_3(\overline{\F}_{q^2}) \setminus \mathcal{O}$ is $\F_{q^2}$-rational if and only if $\al(P_{(a,b)})^2-\al(P_{(a,b)})+1=0$ or its $\PP$-order $i$ satisfies that $i+1$ divides $m$.
\end{lem}

\begin{proof}
If $\alpha:=\alpha(P_{(a,b)})$ has $\PP$-order $i$ for some positive integer $i$, then $\alpha \not \in \{0,1,-\zeta_3,-\zeta_3^2\}$ and $P_{(a,b)} \not \in \mathcal O$. As observed in Remark \ref{rem:PQ:existence_of_i}, we have $\PP_{i+1}(\alpha)=0$ if and only if $((\al+\zeta_3)/(\al+\zeta_3^2))^{3i+3}=1$. If the characteristic $p$ divides $i+1$, we see that $((\al+\zeta_3)/(\al+\zeta_3^2))^{3(i+1)/p}=1$, implying that $\PP_j(\alpha)=0$ for some $j$ strictly smaller than $i+1$. By definition of $\PP$-order, this is impossible. If $\gcd(p,i+1)=1$, the $\alpha$ that have $\PP$-order $i$ are precisely those satisfying that $((\al+\zeta_3)/(\al+\zeta_3^2))^{3}$ is a primitive $(i+1)$-th root of unity. Hence there are $3 \varphi(i+1)$ many $\alpha$ with $\PP$-order $i$. Since $\alpha=a^m/(1+a^m)$ and $\alpha \not \in \{0,1\}$, for each such $\alpha$, there are $m$ distinct possibilities for $a$. Since $P_{(a,b)} \not \in \mathcal{O}$, for each such $a$, there are $q+1$ distinct possibilities for $b$. This proves the first part of the lemma.

Now suppose that $P_{(a,b)} \in  \X_3(\overline{\F}_{q^2}) \setminus \mathcal{O}$ is $\F_{q^2}$-rational and $\alpha^2-\alpha+1 \neq 0$. First of all, we claim that in this case $\alpha \in \F_q$. Indeed, since $a,b \in \F_{q^2}$, we obtain that $a^{3m}=a^{q+1} \in \F_q$ and $a^{2m}+a^m=-b^{q+1} \in \F_q$. But then $a^m=(a^{3m}+a^{2m}+a^m)/(a^{2m}+a^m+1) \in \F_q$. Here we used that $a^{2m}+a^m+1 \neq 0$, which follows from the assumption that $\alpha^2-\alpha+1 \neq 0.$ Now $a^m \in \F_q$, implies $\alpha=a^m/(1+a^m) \in \F_q$. In particular $\al^q=\al$. This implies that
\begin{equation*}
    \left(\frac{\al +\zeta_3}{\al + \zeta_3^2}\right)^{q} = \frac{\al^q +\zeta_3^q}{\al^q + \zeta_3^{2q}} = \frac{\al +\zeta_3^2}{\al + \zeta_3},
\end{equation*}
which is exactly the inverse of $\frac{\al +\zeta_3}{\al + \zeta_3^2}$. Hence $\left(\frac{\al +\zeta_3}{\al + \zeta_3^2}\right)^{3m}=\left(\frac{\al +\zeta_3}{\al + \zeta_3^2}\right)^{q+1} = 1$. This shows that $i+1$ divides $m$.

Conversely, if $\alpha^2-\alpha+1=0$, then $P_{(a,b)} \in \X_3(\overline{\F}_{q^2}) \setminus \mathcal{O}$ satisfies $a^{2m}+a^m+1=0$, which in turn implies $b^{q+1}=1$. Hence $a,b \in \F_{q^2}$. If $\alpha^2-\alpha+1 \neq 0$ and $i+1$ divides $m$, then $3(i+1)$ divides $q+1$ and $\left(\frac{\al +\zeta_3}{\al + \zeta_3^2}\right)^{3(i+1)}=1.$ Hence $\left(\frac{\al +\zeta_3}{\al + \zeta_3^2}\right)^{q+1}=1$, which after clearing denominators amounts to the equation $(\al +\zeta_3)^{q+1}-(\al +\zeta_3^2)^{q+1}=0$. This is a polynomial in $\al$ of degree $q$ and we have already seen that this equation is satisfied for all $\alpha \in \F_q$. We may conclude that $\alpha \in \F_q$. But then $a^m \in \F_q$, which implies that $b^{q+1}=-a^{2m}-a^m \in \F_q$. We conclude that also in this case $P_{(a,b)}$ is an $\F_{q^2}$-rational point of $\X_3$.
\end{proof}

Next, we use the polynomials $\PP_j(s)$ and $\QQ_j(s)$ to investigate the existence of functions that will be useful later when determining gaps at points $P_{(a,b)} \in \X_3(\overline{\F}_{q^2})\setminus \mathcal{O}$.

\begin{thm}
\label{thm:fi}
Let $P_{(a,b)} \in \X_3(\overline{\F}_{q^2})\setminus \mathcal{O}$ and suppose that $\al(P_{(a,b)})^2-\al(P_{(a,b)}) +1 \neq 0$. Further, let $i$ be the $\PP$-order of $\al(P_{(a,b)})$. If $i \le m-2$, then there exists a function $f_{i}\in L((3i+3)D_{\infty})$ such that $v_{P_{(a,b)}}(f_{i})=3i+3$. Moreover, for each $j\in \mathbb{Z}$ with $0\leq j\leq \min\{i-1,m-2\}$, there exists a function $f_j\in L((3j+3)D_{\infty})$ with $v_{P_{(a,b)}}(f_j)=3j+2$.
\end{thm}

\begin{proof}
Throughout the proof we simplify the notation by writing $\al$ instead of $\al(P_{(a,b)})$. In a similar vein, we will write $\mathcal{P}_{j}$ and $\QQ_{j}$, rather than $\mathcal{P}_{j}(\al)$ and $\QQ_{j}(\al)$.

Let $Q_{(A,B)}$ be a point of $\mathcal{H}$ lying over $P_{(a,b)}$ and let $T:=\frac{u-A}{A}$, which is a local parameter at $Q_{(A,B)}$.
For each $j$ such that $0\leq j\leq i$, we claim that there exists a function $f_j\in L((3j+3)D_{\infty})$ such that the local power series expansion of $f_j$ at $Q_{(A,B)}$ with respect to the local parameter $T$ is
\begin{equation}
\label{eq:fj}
f_j = 3\mathcal{P}_{j+1} T^{3j+2} + \QQ_{j+1}  T^{3j+3} + O(T^{q}).
\end{equation}
Note that by definition of the $\PP$-order, this will imply that
\begin{equation}
\label{eq:fi}
f_i = \QQ_{i+1}  T^{3i+3} + O(T^{q}).
\end{equation}
This is sufficient to prove the theorem since, as observed in Remark \ref{rem:unramified}, $v_{Q_{(A,B)}}(f_j)=v_{P_{(a,b)}}(f_j)$ and $3j+3 <q$ for all $j$ under consideration.

First of all, note that, for $j=0$, we can take $f_0$ to be exactly the function defined in equation \eqref{eq:f0} and whose local power series expansion with respect to $T$ was computed in equation \eqref{eq:f0:exp}.
To show the result for $j=1$, we define
\begin{equation*}
    f_1 := -9x_a^2 +27f_0 -3(\al-5)x_a f_0+ (\al^2 -\al -5)f_0^2.
\end{equation*}
Elementary calculations show that the local power series expansion of $f_1$ at $Q_{(A,B)}$ with respect to $T$ is precisely
\begin{equation*}
    f_1 = 3\PP_2 T^5 + \QQ_2 T^6 + O(T^{q}).
\end{equation*}
For $j=2$, we define
\begin{multline*}
f_2:=
(\al+1)^{-3}\left(-27\PP_2f_1+3\PP_2^2f_0^2x_a-3\PP_2(\al^4+\al^3-4\al^2-4\al+3)f_0^3\right)\\+(7\al^2-16\al+7)f_1f_0.
\end{multline*}
A somewhat lengthy, but elementary, calculation shows that the local power series expansion of $f_2$ equals
\begin{equation*}
    f_2 = 3\PP_3 T^5 + \QQ_3 T^6 + O(T^{q}).
\end{equation*}

For $3\leq j \leq i$, we assume now that $f_{j-1}$ and $f_{j-2}$ have the form claimed in equation \eqref{eq:fj} and we construct inductively the remaining functions $f_j$ in the following way, defining:
\begin{equation*}
    f_j:=-\frac{\PP_j f_{j-2}f_1 - \PP_2\PP_{j-1} f_{j-1}f_0}{(\al^2 -\al +1)^2 \PP_{j-2}}.
\end{equation*}
The idea of choosing the functions $f_{j-2}f_1$ and $f_{j-1}f_0$ is that the vanishing order at $Q_{(A,B)}$ is $3j+1$ for both. Hence, a suitable linear combination of them will vanish with order at least $3j+2$. Moreover, as $f_{j-2}f_1$ and $f_{j-1}f_0$ lie in $L((3j+3)D_\infty)$, a linear combination of them does as well. Therefore, we only need to show that
\begin{equation*}
    \PP_j f_{j-2}f_1 - \PP_2\PP_{j-1} f_{j-1}f_0 = -(\al^2 -\al +1)^2 \PP_{j-2}\left(3\mathcal{P}_{j+1} T^{3j+2} + \QQ_{j+1}  T^{3j+3} + O(T^{q})\right).
\end{equation*}

The local power series expansion of $\PP_j f_{j-2}f_1 - \PP_2\PP_{j-1} f_{j-1}f_0$ with respect to $T$ can be obtained from the expansions of the functions $f_{j-2}f_1$ and $f_{j-1}f_0$, which are:
\begin{equation*}
    \begin{aligned}
    f_{j-2}f_1 & = \left(3\PP_{j-1}T^{3j-4} + \QQ_{j-1}  T^{3j-3} + O(T^{q})\right)\left(3\PP_{2}T^{5} + \QQ_{2}  T^{6} + O(T^{q})\right)\\
    & = 9\PP_{2} \PP_{j-1} T^{3j+1} + (3\PP_{j-1}\QQ_{2}  + 3\PP_2\QQ_{j-1})T^{3j+2} + \QQ_2 \QQ_{j-1} T^{3j+3} + O(T^{q}),\\
    f_{j-1}f_0 & = \left(3\PP_{j}T^{3j-1} + \QQ_{j}  T^{3j} + O(T^{q})\right)\left(3T^{2} + \QQ_{1}  T^{3} + O(T^{q})\right)\\
    & = 9\PP_{j}T^{3j+1} + (3\PP_{j}\QQ_{1}  + 3\QQ_{j})T^{3j+2} + \QQ_1 \QQ_{j} T^{3j+3} + O(T^{q}).
    \end{aligned}
\end{equation*}
Hence, we have
\begin{equation*}
\begin{split}
    \PP_j f_{j-2}f_1 - \PP_2\PP_{j-1} f_{j-1}f_0 = 3(\PP_{j-1} \PP_j \QQ_2  + \PP_2 \PP_j \QQ_{j-1}  - \PP_2 \PP_{j-1} \PP_j \QQ_1  - \PP_2 \PP_{j-1} \QQ_{j} )T^{3j+2}\\
     + \left(\PP_j \QQ_{j-1}  \QQ_2  - \PP_{j-1} \PP_2 \QQ_{j}  \QQ_1 \right)T^{3j+3} + O(T^{q}).
\end{split}
\end{equation*}
We are therefore left to prove the two following identities:
\begin{equation}
\label{eq:coeff:T3i+2}
    3\left(\PP_{j-1} \PP_j \QQ_2  + \PP_2 \PP_j \QQ_{j-1}  - \PP_2 \PP_{j-1} \PP_j \QQ_1  - \PP_2 \PP_{j-1} \QQ_{j} \right)
    = -3(\al^2 -\al +1)^2 \PP_{j-2} \PP_{j+1}
\end{equation}
and
\begin{equation}
\label{eq:coeff:T3i+3}
    \PP_j \QQ_{j-1}  \QQ_2  - \PP_{j-1} \PP_2 \QQ_{j}  \QQ_1  = -(\al^2 -\al +1)^2 \PP_{j-2} \QQ_{j+1} .
\end{equation}
This can be conveniently done by using Lemma \ref{lem:PQ:identities}. Indeed, consider first equation \eqref{eq:coeff:T3i+2} and use identity \eqref{eq:id:2} as
\begin{equation*}
    \PP_j \QQ_{2} -\PP_2 \QQ_{j} =\PP_{j-2} \QQ_{0} \cdot (s^2-s+1)^{6},
\end{equation*}
i.e., with indices $(j,2,0)$ (listed in order as in the statement of Lemma \ref{lem:PQ:identities}).
Then, we obtain
\begin{equation}
\label{eq:first:part:1}
\begin{split}
    \PP_{j-1} \PP_j \QQ_2  - \PP_2 \PP_{j-1} \QQ_{j} &= \PP_{j-1}(\PP_j\QQ_2  - \PP_2\QQ_{j} )\\
    &= \PP_{j-1}\cdot(\al^2 - \al + 1)^5 \PP_{j-2}.
\end{split}
\end{equation}
By using again equation \eqref{eq:id:2}, this time with indices $(j-1,1,0)$,
we can also rewrite
\begin{equation}
\label{eq:second:part:1}
\begin{split}
    \PP_2 \PP_j \QQ_{j-1}  - \PP_2 \PP_{j-1} \PP_j \QQ_1  &= \PP_2 \PP_j \QQ_{j-1}  \PP_1 - \PP_2 \PP_{j-1} \PP_j \QQ_1 \\
    &= -\PP_2\PP_j (\PP_{j-1}\QQ_1  - \PP_1\QQ_{j-1} )\\
    &= -\PP_2\PP_j \cdot (\al^2 - \al + 1)^2 \PP_{j-2}.
\end{split}
\end{equation}
Then, by equations \eqref{eq:first:part:1} and \eqref{eq:second:part:1}, we have that equation \eqref{eq:coeff:T3i+2} is equivalent to
\begin{equation*}
    \PP_{j-1}\cdot (\al^2 - \al + 1)^5 \PP_{j-2}-\PP_2\PP_j \cdot (\al^2 - \al + 1)^2 \PP_{j-2} = -(\al^2 -\al +1)^2 \PP_{j-2} \PP_{j+1}.
\end{equation*}
Dividing out the factor $(\al^2 - \al + 1)^2 \PP_{j-2}$ both in the right hand side and the left hand side of this equality and rearranging the terms, we obtain
\begin{equation*}
    \PP_{j-1}\cdot (\al^2 - \al + 1)^3 = \PP_2 \PP_j -\PP_{j+1},
\end{equation*}
which holds by Lemma \ref{lem:PQ:identities}, as it is precisely identity \eqref{eq:id:1} with indices $(j,1,1)$.

In order to prove equation \eqref{eq:coeff:T3i+3}, we can argue in a similar way. Indeed, we have:
\begin{multline*}
    \PP_j \QQ_{j-1}  \QQ_2  - \PP_{j-1} \PP_2 \QQ_{j}  \QQ_1  =
    (\PP_j \QQ_2  - \PP_2 \QQ_{j}  + \PP_2 \QQ_{j} )\QQ_{j-1}  - \PP_{j-1} \PP_2 \QQ_{j}  \QQ_1 \\
    =\left(\PP_{j-2}\cdot (\al^2 -\al + 1)^5 + \PP_2\QQ_{j} \right)\QQ_{j-1}  - \PP_{j-1} \PP_2 \QQ_{j}  \QQ_1,
\end{multline*}
where the last equality follows from equation \eqref{eq:id:2} with indices $(j,2,0)$.
Moreover,
\begin{multline*}
    \left(\PP_{j-2}\cdot (\al^2 -\al + 1)^5 + \PP_2\QQ_{j} \right)\QQ_{j-1}  - \PP_{j-1} \PP_2 \QQ_{j}  \QQ_1  = \\
    \PP_{j-2}\cdot (\al^2 -\al + 1)^5\QQ_{j-1}  - \PP_2\QQ_{j} \left(\PP_{j-1}\QQ_1  - \PP_1 \QQ_{j-1} \right) = \\
    \PP_{j-2}\cdot (\al^2 -\al + 1)^5\QQ_{j-1}  - \PP_2 \QQ_{j} \PP_{j-2}\cdot (\al^2-\al+1)^2,
\end{multline*}
where the last equality follows from equation \eqref{eq:id:2} with indices $(j-1,1,0)$.
Finally, using again equation \eqref{eq:id:2} with indices $(2,1,j-1)$,
we have
\begin{multline*}
    \PP_{j-2}\cdot(\al^2 -\al + 1)^5\QQ_{j-1}  - \PP_2 \QQ_{j} \PP_{j-2}\cdot(\al^2-\al+1)^2 = \\
    -\PP_{j-2}\cdot(\al^2 -\al + 1)^2\left(\PP_2 \QQ_{j}  - \QQ_{j-1}\cdot (\al^2-\al + 1)^3 \PP_1\right) =
     -\PP_{j-2}\cdot(\al^2 -\al + 1)^2 \QQ_{j+1},
\end{multline*}
which proves equation \eqref{eq:coeff:T3i+3}.

From this, equation \eqref{eq:fj} follows directly, while equation \eqref{eq:fi} follows observing that $\PP_{i+1} = 0$ by hypothesis and $\QQ_{i+1} \neq 0$ by Remark \ref{rem:PQ:nocommon:roots}.
As we have already observed that, by construction, $f_j\in L((3j+3)D_{\infty})$ for all $j$ in $0\leq j\leq i$, the proof of the theorem is then completed.
\end{proof}

The proof of the theorem does not work if $\alpha^2-\alpha+1=0$, but a very similar approach works as will become clear in the proof of the following result. Recall that, if $\alpha^2-\alpha+1=0$, then $\al$ is not a root of any $\PP_{i}$, for all $i\in \mathbb{Z}_{>0}$.

\begin{thm}
\label{thm:gi}
Suppose that $P_{(a,b)}$ is a point on $\X_3$ such that $\al(P_{(a,b)})^2-\al(P_{(a,b)})+1=0$. Then, for every positive integer $i$ such that $i \le m-2$, there exists a function $g_i\in L((3i+3)D_{\infty})$ with $v_{P_{(a,b)}}(g_i)=3i+2$.
\end{thm}

\begin{proof}
As before, in this proof we write $\al$ instead of $\al(P_{(a,b)})$ and $\mathcal{P}_{j}$, $\QQ_{j}$ instead of $\mathcal{P}_{j}(\al)$, $\QQ_{j}(\al)$. For each $i\in \mathbb{Z}_{\geq 0}$, we claim that there exists a function $g_i \in L((3i+3)D_\infty)$ such that the local power series expansion of $g_i$ at $Q_{(A,B)}$ with respect to the local parameter $T$ is:
\begin{equation}\label{eq:gi}
g_i = 3T^{3i+2} + (\alpha+1)T^{3i+3} + O(T^{q})
\end{equation}
Denoting by $f_0$ and $f_1$, the functions constructed in the previous theorem, we see that $g_0=f_0$, while $g_1=(2\alpha-1)f_1/9$, since $$(2\alpha-1)3\PP_2 \equiv 27 \pmod{\alpha^2-\alpha+1}$$ and $$(2\alpha-1)\QQ_2 \equiv 9\alpha+9 \pmod{\alpha^2-\alpha+1}.$$

For $i \geq 2$, we assume now that $g_{i-1}$ and $g_{i-2}$ have the form claimed in equation \eqref{eq:gi} and we construct inductively the remaining functions $g_i$ by taking a suitable linear combination of
\begin{equation*}
    g_{i-1}, \quad g_{i-2}\cdot g_0\cdot x_a, \quad g_{i-2}\cdot g_0^2 \quad \mbox{and} \quad g_{i-1}\cdot g_0.
\end{equation*}
The point of choosing these four functions, is that their vanishing orders at $Q_{(A,B)}$ are $3i-1$, $3i-1$, $3i$ and $3i+1$ respectively. Therefore a suitable linear combination of them will vanish with order at least $3i+2$. Moreover, since the four function all lie in $L((3i+3)D_\infty)$, any linear combination of them does as well.

More in detail, if we set
\[g_i:=(6\alpha-3)g_{i-1}-\frac{2\alpha-1}{3}g_{i-2}g_0x_a+\frac{3\alpha-2}{3}g_{i-2}g_0^2-(\alpha-2) g_{i-1}g_0,\]
then a direct computation shows that equation \eqref{eq:gi} is satisfied.
\end{proof}

\section{Weierstrass semigroups at the $\F_{q^2}$-rational points of $\X_3$}
\label{sec:rational}

In this section, we compute the Weierstrass semigroups at all the $\F_{q^2}$-rational points of $\X_3$. We start with the determination of the semigroup at the points of the set $\mathcal{O}$ and then continue to all other $\F_{q^2}$-rational points of $\X_3$. We will assume that $q$ is at least five, so that $m \ge 2$. If $q=2$, the curve $\X_3$ is an elliptic curve, so all Weierstrass semigroups are just $\{0\} \cup \mathbb{Z}_{\ge 2}$ in that case. 

\begin{rem}
    By the Fundamental Equation (\cite{HKT}*{Page xix (ii)}) and by \cite{HKT}*{Proposition 10.9}, it is well known that both $q$ and $q+1$ are non-gaps at every $\F_{q^2}$-rational point of an $\F_{q^2}$-maximal curve. However, in Theorem \ref{thm:O} and in Lemma \ref{lem:q:and:q+1}, we prove this fact again in the particular case of $\X_3$, as we show this with some easy explicit computations.
\end{rem}

\subsection{The Weierstrass semigroup at $P\in \mathcal{O}$}

\begin{thm}
\label{thm:O}
Let $P\in \mathcal{O}$.
Then $H(P) = \langle q-2, q, q+1\rangle.$
\end{thm}
\begin{proof}
We will prove that
\begin{equation*}
    H(P_{(a,0)})= \langle q-2, q, q+1\rangle
\end{equation*}
for $P_{(a,0)}\in \mathcal{O}$ a point such that $a^m+1=0$, and hence the result will follow as, by Corollary \ref{cor:orbitO}, $\mathcal{O}$ is contained in an orbit of $\mathrm{Aut}(\X_3)$ and all the points in the same orbit have the same Weierstrass semigroup.

We start by showing that the semigroup $H:=\langle q-2, q, q+1\rangle$, that is to say, the semigroup generated by $q-2,q$ and $q+1$, is contained in $H(P_{(a,0)})$. Proposition \ref{prop:divisors-on-X3} implies that the functions
\begin{equation*}
        \frac{1}{x-a}, \quad \frac{y}{x-a}, \quad \mbox{and} \quad \frac{y^3}{x(x-a)}
\end{equation*}
in $\F_{q^2}(\X_3)$ only have a pole at $P_{(a,0)}$ and of order $q+1$, $q$, and $q-2$ respectively.
This shows that $q-2,q,q+1\in H(P_{(a,0)})$, proving that $H\subseteq H(P_{(a,0)})$.

Hence to conclude the proof of the theorem it is sufficient to show that the number of gaps of semigroup $H$, also known as the genus of $H$, is equal to $g(\X_3)$. To do so, note that semigroup $H$ is telescopic, since the sequence $(a_1,a_2,a_3):=(q-2,q+1,q)$ is a telescopic sequence. See for example \cite{HLP}*{Section 5.4} for a short discussion on telescopic semigroups. Defining $d_0=0$, $d_1=q-2$, $d_2=\gcd(q-2,q+1)=3$, and $d_3=\gcd(q-2,q,q+1)=1$, the genus of $H$ is according to \cite{HLP}*{Proposition 5.35} given by
 \begin{equation*} \label{gentelescopic}
g(H)=\frac{1}{2}\bigg( 1+\sum_{i=1}^{3} \bigg( \frac{d_{i-1}}{d_i}-1 \bigg) a_i \bigg)=g(\X_3).
\end{equation*}
\end{proof}

\subsection{The Weierstrass semigroups at the points in $\X_3(\F_{q^2})\setminus \mathcal{O}$}

\begin{lem} \label{lem:q:and:q+1}
Let $P_{(a,b)}\in \X_3(\F_{q^2})\setminus\mathcal{O}$. Then $q+1$ and $q$ are contained in $H(P_{(a,b)})$.
\end{lem}

\begin{proof}
The fact that $q+1 \in H(P_{(a,b)})$ is simply a consequence of equation \eqref{fundeq}. To prove that $q \in H(P_{(a,b)})$, let $P_{(\bar{a},0)}\in \X_3(\F_{q^2})$ such that $\bar{a}^m+1=0$ and consider the function

$$h_q:=\frac{(x-a)f_{P_{(\bar{a},0)},1}}{f_{P_{(a,b)},1}(x-\bar{a})},$$

where the functions $f_{P_{(a,b)},1}$ and $f_{P_{(\bar{a},0)},1}$ are defined as in equation \eqref{fundeq}. Then, from equations \eqref{fundeq}, \eqref{eq:div:xa}, \eqref{eq:divs}, one has
\begin{equation*}
(h_q)=-qP_{(a,b)} +\sum_{\xi^{q+1}=1, \ \xi \ne 1} P_{(a,\xi b)},
\end{equation*}
implying that $q \in H(P_{(a,b)})$.
\end{proof}

\begin{thm}
\label{thm:special:short:orbit}
Let $P_{(a,b)}\in \X_3(\F_{q^2})\setminus\mathcal{O}$ be a point such that $\alpha(P_{(a,b)})^2-\alpha(P_{(a,b)})+1=0$. Then
\begin{equation*}
    H(P_{(a,b)}) = \langle q, q+1, (q-1) +i(q-2) \mid i=0,\ldots, m-2\rangle.
\end{equation*}
\end{thm}
\begin{proof}
We start by showing that the semigroup $H:=\langle q, q+1, (q-1) +i(q-2) \mid i=0,\ldots, m-2\rangle$ is contained in $H(P_{(a,b)})$. To this aim, we show that $q, q+1, (q-1) +i(q-2)$, for all $i=0,\ldots, m-2$, are pole numbers of $P$. By Lemma \ref{lem:q:and:q+1}, we already know that $q,q+1\in H(P_{(a,b)})$, so we are left to show that $(q-1) +i(q-2)$ is a pole number for every $i=0,\ldots, m-2$. We prove this considering the following family of functions. For all $i$ such that $0\leq i\leq m-2$, let $P_{(\bar{a},0)}$ be a point such that $\bar{a}^m+1=0$ and define the function
\begin{equation*}
    G_i:= \frac{g_i \cdot f_{P_{(\bar{a},0)},1}^{i+1}}{f_{P_{(a,b)},1}^{i+1} \cdot (x-\bar{a})^{i+1}},
\end{equation*}
where the functions $g_i$ are those built in Theorem \ref{thm:gi}.
Then using equations \eqref{fundeq}, \eqref{eq:div:xa}, \eqref{eq:divs} and Theorem \ref{thm:gi}, the divisor of the function $G_i$ is seen to be
\begin{equation*}
(G_i) =  E_i - ((q-1) + i(q-2))P_{(a,b)},
\end{equation*}
where $E_i\in \mathrm{Div}(\X_3)$ is an effective divisor such that $P_{(a,b)}\not \in \mathrm{supp}(E_i)$.
Therefore, $(q-1) +i(q-2)\in H(P_{(a,b)})$ for all $i=1,\ldots, m-2$.

To complete the proof, we need to show that the genus of the semigroup $H$ is less than or equal to $g(\X_3)$. Indeed the inequality $g(H) \ge g(\X_3)$ is already clear, since we just showed that $H \subseteq H(P_{(a,b)})$. Of course we know $0 \in H$, but we claim that for $j=1,\dots,m-1$, all integers in $\{j(q-2)+1,\dots,j(q+1)\}$ are in $H$ as well. This is clear for $j=1$, since $q-1,q,q+1 \in H$. If this is true for some $j<m-1$, then adding $q-1$ and $q+1$ to all integers in $\{j(q-2)+1,\dots,j(q+1)\}$, shows that the consecutive integers in $\{(j+1)(q-2)+2,\dots,(j+1)(q+1)\}$ are all in $H$. Since $(j+1)(q-2)+1=(q-1)+j(q-2) \in H$, we conclude that all integers in $\{(j+1)(q-2)+1,\dots,(j+1)(q+1)\}$ are in $H$. This shows the claim. Now note that $\{(m-1)(q-2)+1,\dots,(m-1)(q+1)\}$ consists of $q-2$ consecutive integers, all in $H$. Adding integral multiples of $q-1$ and $q$ to this set, we obtain that all integers greater than or equal to $(m-1)(q-2)+1+q-1=(m-1)(q+1)+2$ are in $H$. This means that the number of gaps in $H$ is at most \[g(H) \le (q-2)+(q-5)+\cdots+3+1.\] The final $+1$ counts the potential gap $(m-1)(q+1)+1$. Hence \[g(H) \le 1+3\sum_{k=1}^{m-1} k =1+3 m(m-1)/2 =g(\X_3),\]
which is what we needed to show.
\end{proof}

\begin{thm}
\label{thm:rationalcase:other:semigroups}
Let $P_{(a,b)}\in \X_3(\F_{q^2}) \setminus \mathcal{O}$ be a point such that $\alpha(P_{(a,b)})^2-\alpha(P_{(a,b)})+1 \neq 0$. Further, let $i$ be the $\PP$-order of $\alpha(P_{(a,b)})$. If $i \le m-2$, then
\begin{equation*}
    H(P_{(a,b)}) = \langle q, q+1, (q-1) +j(q-2),(q-1)+i(q-2)-1 \mid j=0,\ldots, i-1\rangle.
\end{equation*}
If $i = m-1$, then
\begin{equation*}
    H(P_{(a,b)}) = \langle q, q+1, (q-1) +j(q-2) \mid j=0,\ldots, m-2\rangle.
\end{equation*}
\end{thm}
\begin{proof}
We first assume that $i \le m-2$.
We proceed similarly as in the proof of Theorem \ref{thm:special:short:orbit}, showing that the semigroup $H:=\langle q-1, q, q+1, (q-1) +j(q-2), (q-1)+i(q-2)-1 \mid j=1,\ldots, i-1\rangle$ is contained in $H(P_{(a,b)})$ and has at most $g(\X_3)$ gaps. For all $j$ such that $j=1,\ldots, i-1$, let $P_{(\bar{a},0)}$ be a point with $\bar{a}^m+1=0$ and define the function
\begin{equation*}
    F_j:= \frac{f_j \cdot f_{P_{(\bar{a},0)},1}^{j+1}}{f_{P_{(a,b)},1}^{j+1} \cdot (x-\bar{a})^{j+1}},
\end{equation*}
where the $f_j$ are the functions built in Theorem \ref{thm:fi}.
Using equations \eqref{fundeq}, \eqref{eq:div:xa}, \eqref{eq:divs} and Theorem \ref{thm:gi}, the divisor of the function $F_j$ can be seen to be
\begin{equation*}
    (F_j) = E_j - ((q-1) + j(q-2))P_{(a,b)}
\end{equation*}
where $E_j\in \mathrm{Div}(\X_3)$ is an effective divisor such that $P_{(a,b)}\not \in \mathrm{supp}(E_j)$.
Therefore, $(q-1) +j(q-2)\in H(P)$ for all $j=1,\ldots, i-1$.
Similarly
\[    (F_i) = E_i - ((q-1) + i(q-2)-1)P_{(a,b)},\]
where $E_{i}\in \mathrm{Div}(\X_3)$ is an effective divisor such that $P_{(a,b)}\not \in \mathrm{supp}(E_{i})$.
Hence, we now have shown that $H\subseteq H(P_{(a,b)})$.

What remains to be shown is that the genus of the semigroup $H$ does not exceed $g(\X_3)$. We know $0 \in H$ and just as in the proof of Theorem \ref{thm:special:short:orbit} we conclude that all integers in the set $\{j(q-2)+1,\dots,j(q+1)\}$ are in $H$ for any $j=1,\dots,i$. Further, we have already shown that $(i+1)(q-2) \in H$ and adding $q-1$, $q$, and $q+1$ to the integers in $\{i(q-2)+1,\dots,i(q+1)\}$ yields that $\{(i+1)(q-2)+2,\dots,(i+1)(q+1)\} \subseteq H$.

Since $P_{(a,b)} \in \X_3(\F_{q^2})$, Lemma \ref{lem:Porder_rational} implies that $i+1$ divides $m$. We claim that for $k=0,\dots,m/(i+1)-1$ and all $j=1,\dots,i$ the sets $\{(k(i+1)+j)(q-2)+1,\dots,(k(i+1)+j)(q+1)\}$ are contained in $H$ as well as the integer $((k+1)(i+1))(q-2)$ and the set $\{(k+1)(i+1)(q-2)+2,\dots,(k+1)(i+1)(q+1)\}$. We have so far shown this for $k=0$. If the claim is true for some $k-1 < m/(i+1)-1$, adding $(i+1)(q-2)$ and the integers in $\{(i+1)(q-2)+2,\dots,(i+1)(q+1)\}$, immediately shows that the claim is true for $k$ as well. This proves the claim. For $k=m/(i+1)-1$, we obtain that $\{m(q-2)+2,\dots,m(q+1)\}$, which contains $q$ consecutive integers, is a  subset of $H$. This shows that all integers greater than or equal to $m(q-2)+2=(m-1)(q+1)+2$ are in $H$. Estimating the number of gaps is now done very similarly as in the proof of Theorem \ref{thm:special:short:orbit}. The number of gaps of the semigroup there is in fact exactly the same as those of the semigroup $H$ constructed here: in the proof of Theorem for all $k=1,\dots,m/(i+1)-1$, the integer $k(i+1)(q-2)$ was a gap, while $k(i+1)(q-2)+1$ was not, while now $k(i+1)(q-2)$ is in $H$ and $k(i+1)(q-2)+1$ is not. Hence $g(H) \le g(\X_3)$ again holds.

We are left to prove the theorem if $i=m-1$. Using exactly the same approach as above, we can show that $H:=\langle q-1, q, q+1, (q-1) +j(q-2) \mid j=1,\ldots, m-2\rangle$ is contained in $H(P_{(a,b)})$. Now note that $H$ is exactly the same semigroup as the one occurring in Theorem \ref{thm:special:short:orbit}. Hence $g(H) \le g(\X_3)$ holds in this case as well.
\end{proof}

\subsection{Some remarks on rational Weierstrass points}

From the previous two subsections, we have a complete determination of all types of Weierstrass semigroups that occur among them and how many points attain a given type. To avoid trivial cases, we assume $q \ge 5$.
\begin{thm}\label{thm:summinguprational}
The number of distinct Weierstrass semigroups $H(P)$ among $P \in \X_3(\F_{q^2})$ is exactly the same as the number of divisors of $m$. The semigroups that occur and the $P\in \X_3(\F_{q^2})$ for which they occur are:
\begin{itemize}
\item $H(P)= \langle q-2,q,q+1 \rangle$ for $q+1$ many $P \in {\mathcal O}$.
\item $H(P) = \langle q, q+1, (q-1) +j(q-2),(q-1)+i(q-2)-1 \mid j=0,\ldots, i-1\rangle$, where $1 \le i \le m-2$ and $i+1$ divides $m$, for the $(q+1)^2\varphi(i+1)$ many $P \in \X_3(\F_{q^2})$ for which $\alpha(P)$ has $\PP$-order $i$.
\item $H(P) = \langle q, q+1, (q-1) +j(q-2) \mid j=0,\ldots, m-2\rangle$ for the $(q+1)^2\varphi(m)$ many $P \in \X_3(\F_{q^2})$ for which $\alpha(P)$ has $\PP$-order $m-1$ as well as for the $2m(q+1)$ many $P \in \X_3(\F_{q^2})$ for which $\al(P)^2-\al(P)+1=0$.
\end{itemize}
\end{thm}
\begin{proof}
First of all, Theorems \ref{thm:O}, \ref{thm:special:short:orbit} and \ref{thm:rationalcase:other:semigroups} combined describe all possible Weierstrass semigroups that occur among $P \in \X_3(\F_{q^2}).$ Lemma \ref{lem:Porder_rational} implies that the only possible $\PP$-orders $i$ for $P \in \X_3(\F_{q^2})$ correspond to divisors $i+1 \ge 2$ of $m$. Therefore the total number of possible Weierstrass semigroups is exactly the number of divisors of $m$, where the divisor $1$ counts the semigroup $\langle q-2,q,q+1 \rangle.$

As for the number of $P \in \X_3(\F_{q^2})$ attaining one particular type: we know that $|\mathcal O|=q+1$, while Lemma \ref{lem:Porder_rational} implies how many $P$ have $\PP$-order equal to a given $i$. The only number of points left to determine is those $P \in \X_3(\F_{q^2})$ such that $\al(P)^2-\al(P)+1=0$. Using that $\al(P_{(a,b)})=a^m/(1+a^m),$ we see that $\al(P_{(a,b)})^2-\al(P_{(a,b)})+1=0$ if and only if $a^{2m}+a^m+1=0$ and $b^{q+1}=-1$. Hence, for exactly $2m(q+1)$ many $P \in \X_3(\F_{q^2})$ one has $\al(P)^2-\al(P)+1=0$.
\end{proof}

\begin{rem}
It is not hard to see that the indicated generators in Theorem \ref{thm:summinguprational} are in all cases a minimal set of generators. Since $\X_3$ is a maximal curve over $\F_{q^2}$, its number of $\F_{q^2}$-rational points is equal to $q^2+1+2qg(\X_3)=\frac{(q+1)(q^2+q+3)}{3}$. Note as a sanity check that indeed, $$q+1+ \sum_{i=1;i+1|m}^{m-1} (q+1)^2 \varphi(i+1) + 2m(q+1) = q^2+1+2qg(\X_3),$$ using the equation $\sum_{d|m} \varphi(d)=m$ where the sum is over all divisors of $m$.

Also the multiplicity (i.e., the smallest positive element) of the semigroups is easy to determine using Theorem \ref{thm:summinguprational}: it is $q-1$, unless $P \in {\mathcal O}$ in which case it is $q-2$. Another parameter of a numerical semigroup is its conductor $c$. This is the smallest nonnegative integer $c$ such that $\mathbb{Z}_{ \ge c}$ is contained in the semigroup. Since $(q-2)D_\infty$ is a canonical divisor by Lemma \ref{lem:canonicaldiv}, $3D_\infty \sim (q+1)P_{(a,0)}$ by equation \eqref{eq:div:xa}, and for any $P \in \X_3(\F_{q^2})$ by the Fundamental Equation $(q+1)P \sim (q+1)P_{(a,0)}$, we see that $(q-2)mP$ is a canonical divisor for all $P \in \X_3(\F_{q^2})$. This implies that $H(P)$ is symmetric for all $P \in \X_3(\F_{q^2})$. In particular the largest gap in $H(P)$ is $2g(\X_3)-1$ for all $P \in \X_3(\F_{q^2})$, implying that the conductor of $H(P)$ is $2g(\X_3)$.
\end{rem}

\section{Weierstrass semigroups at the non-$\F_{q^2}$-rational points of $\X_3$}
\label{sec:nonrational}
In this section, we compute the Weierstrass semigroups at all the remaining points of $\X_3$, namely at all the non-$\F_{q^2}$-rational points. We start by computing the semigroup for the generic case, i.e., for the non-Weierstrass points of the curve and, finally, we determine the semigroups for the non-$\F_{q^2}$-rational Weierstrass points.

\subsection{The generic case}

\begin{thm} \label{maingapsgeneric}
Let $P_{(a,b)} \in \mathcal{X}_3(\overline{\mathbb{F}}_{q^2}) \setminus \mathcal{X}_3(\mathbb{F}_{q^2})$ such that $\PP_{j}(\al) \ne 0$ for all $j=2,\ldots,m-1$. Then
$$G(P_{(a,b)})=\{jq+k \mid j=0,\ldots,m-2, \ k=1,\ldots,q-3j-2\}\cup \{(m-1)q+1\},$$
that is
$$H(P_{(a,b)})=\{0, (j+1)(q-3)+2+k, (m-1)q+2,\ldots \mid j=0,\ldots, m-2, \ k=0,\ldots,3j+1\}.$$
\end{thm}

\begin{proof}
Let $G:=\{jq+k \mid j=0,\ldots,m-2, \ k=1,\ldots,q-3j-2\}\cup \{(m-1)q+1\}$ be the putative set of gaps. Direct computations show that $|G|=1+\sum_{j=0}^{m-2}(q-(3j+2))=g(\mathcal{X}_3)$.

We need to prove that, for every $g\in G$, there exists a function $h_g\in L((q-2)D_\infty)$ such that $v_P(h_g)=g-1$.

Let $g=jq+k\in G$. We distinguish the following cases.
\begin{enumerate}
    \item  If $\left\lfloor\frac{k}{3} \right\rfloor \neq 0$, then we define:
    \begin{equation*}
h_g := \begin{cases} F_{P_{(a,b)}}^j \cdot f_{\left\lfloor\frac{k}{3} \right\rfloor-1} \ &\mbox{if} \ k \equiv 0 \pmod 3,\\[15pt]
F_{P_{(a,b)}}^j \cdot (y-b)\cdot f_{\left\lfloor\frac{k}{3} \right\rfloor-1} \ &\mbox{if} \ k \equiv 1 \pmod 3,\\[15pt]
F_{P_{(a,b)}}^j \cdot t_{P_{(a,b)}}\cdot f_{\left\lfloor\frac{k}{3} \right\rfloor-1} \ &\mbox{if} \ k \equiv 2 \pmod 3.\\
\end{cases}
    \end{equation*}
    \item If $\left\lfloor\frac{k}{3} \right\rfloor = 0$, we define instead:
        \begin{equation*}
h_g := \begin{cases}
F_{P_{(a,b)}}^j \ &\mbox{if} \ k = 1,\\[15pt]
F_{P_{(a,b)}}^j \cdot (y-b) \ &\mbox{if} \ k = 2.\\
\end{cases}
    \end{equation*}
\end{enumerate}
Here, the function $f_{\left\lfloor\frac{k}{3} \right\rfloor-1}$ is one of the functions $f_i$ constructed in Theorem \ref{thm:fi} and the function $F_{P_{(a,b)}}$ is as defined in equation \eqref{eq:F}.

Note that, as $j=0,\ldots,m-2$, for $k=3,\ldots,q-3j-2$ it holds that
\begin{equation*}
        0 \leq \left\lfloor\frac{k}{3} \right\rfloor-1 \leq \left\lfloor\frac{q-3j-2}{3} \right\rfloor-1 = \frac{q-2}{3} -j -1 = m-2-j\leq m-2,
\end{equation*}
hence the function $h_g$ is well-defined, for any $i=0,\ldots,m-2$ and $k=1,\ldots,q-3j-2$.
Indeed, defining the function $h_g$ in this way, for any $g=jq+k\in G$, we have what follows.

\noindent
{\bf Case 1:} $\left\lfloor\frac{k}{3} \right\rfloor \neq 0$.

If $k \equiv 0 \pmod 3$, then
    $$v_P(h_{g})=jq+3\bigg(\left\lfloor\frac{k}{3} \right\rfloor-1\bigg)+2= jq + k -1$$
and
    $$(h_{g})_\infty\leq \left(3j+3\bigg(\left\lfloor\frac{k}{3} \right\rfloor-1+1\bigg)\right)D_\infty = (3j+k)D_\infty\leq (3j+q-3j-2)D_\infty = (q-2)D_\infty.$$

If $k \equiv 1 \pmod 3$, then
    $$v_P(h_{g})=jq+3\bigg(\left\lfloor\frac{k}{3} \right\rfloor-1\bigg)+2+1= jq + (k -1) -3 +3 = jq + k -1$$
and
\begin{eqnarray*}
(h_{g})_\infty & \leq & \left(3j+3\bigg(\left\lfloor\frac{k}{3} \right\rfloor-1+1\bigg)+2\right)D_\infty = (3j+k+1)D_\infty\\
 & \leq & (3j+q-3j-4+1)D_\infty = (q-3)D_\infty,
\end{eqnarray*}
    where the last inequality follows from the fact that $q-3j-2 \equiv 0 \pmod 3$, hence if $k\equiv 1 \pmod 3$, then $k\leq (q-3j-2)-2=q-3j-4$.

If $k \equiv 2 \pmod 3$, then
    $$v_P(h_{g})=jq+3\bigg(\left\lfloor\frac{k}{3} \right\rfloor-1\bigg)+2+2= jq + (k -2) -3 +4 = jq + k -1$$
and
\begin{eqnarray*}
(h_{g})_\infty & \leq & \left(3j+3\bigg(\left\lfloor\frac{k}{3} \right\rfloor-1+1\bigg)+3\right)D_\infty = (3j+k+1)D_\infty \\
 & \leq & (3j+q-3j-3+1)D_\infty = (q-2)D_\infty,
\end{eqnarray*}

    where the last inequality follows from the fact that $q-3j-2 \equiv 0 \pmod 3$, hence if $k\equiv 2 \pmod 3$, then $k\leq (q-3j-2)-1=q-3j-3$.

\noindent
{\bf Case 2:} $\left\lfloor\frac{k}{3} \right\rfloor = 0$.

If $k=1$, then
        $$v_P(h_{g})=jq$$
and
    $$(h_{g})_\infty\leq (3j)D_\infty \leq ((q+1) - 6)D_\infty = (q-5)D_\infty.$$

If $k=2$, then
        $$v_P(h_{g})=jq + 1$$
        and
        $$(h_{g})_\infty\leq (3j + 2)D_\infty \leq ((q+1) - 6 + 2)D_\infty = (q-3)D_\infty.$$
\end{proof}

Since the Weierstrass semigroup at all but a finite number of points of $\X_3$ is as described in Theorem \ref{maingapsgeneric}, we call
\begin{equation*}
    H_{gen}:=\{0, (j+1)(q-3)+2+k, (m-1)q+2,\ldots \mid j=0,\ldots, m-2, \ k=0,\ldots,3j+1\}
\end{equation*}
the \emph{generic} Weierstrass semigroup of $\X_3$ and
\begin{equation*}
    G_{gen}:=\{jq+k \mid j=0,\ldots,m-2, \ k=1,\ldots,q-3j-2\}\cup \{(m-1)q+1\}
\end{equation*}
the \emph{generic} set of gaps of $\X_3$.

\subsection{The Weierstrass semigroups at the non-$\F_{q^2}$-rational Weierstrass points}

\begin{thm}
\label{thm:non:rational:weierstrass}
Let $P_{(a,b)} \in \mathcal{X}_3(\overline{\mathbb{F}}_{q^2}) \setminus \mathcal{X}_3(\mathbb{F}_{q^2})$ and $i$ the $\PP$-order of $\alpha(P_{(a,b)})$. Suppose that $i \le m-2$.
Then
\begin{equation}
\label{eq:gaps:nonrational:some:i}
\begin{split}
    G(P_{(a,b)})= \left(G_{gen} \setminus \left \{(m-2-i-\ell(i+1))q + (\ell+1)(3i+3) \mid \ell = 0,\ldots, \left\lfloor \frac{m-2-i}{i+1}\right \rfloor\right \} \right) \\ \cup \left \{(m-2-i-\ell(i+1))q + (\ell+1)(3i+3)+1 \mid \ell = 0,\ldots, \left\lfloor \frac{m-2-i}{i+1}\right \rfloor\right \},
\end{split}
\end{equation}
that is
\begin{equation*}
\begin{split}
    H(P_{(a,b)})= \left(H_{gen} \setminus \left \{(m-2-i-\ell(i+1))q + (\ell+1)(3i+3)+1 \mid \ell = 0,\ldots, \left\lfloor \frac{m-2-i}{i+1}\right \rfloor\right \} \right) \\ \cup \left \{(m-2-i-\ell(i+1))q + (\ell+1)(3i+3) \mid \ell = 0,\ldots, \left\lfloor \frac{m-2-i}{i+1}\right \rfloor\right \}.
\end{split}
\end{equation*}
\end{thm}

\begin{proof}
Let $G$ as in equation \eqref{eq:gaps:nonrational:some:i} be the putative set of gaps.
Since the cardinality of the set
\begin{equation*}
    \left \{(m-2-i-\ell(i+1))q + (\ell+1)(3i+3) \mid \ell = 0,\ldots, \left\lfloor \frac{m-2-i}{i+1}\right \rfloor\right \}
\end{equation*}
is the same as the cardinality of the set
\begin{equation*}
    \left \{(m-2-i-\ell(i+1))q + (\ell+1)(3i+3) + 1 \mid \ell = 0,\ldots, \left\lfloor \frac{m-2-i}{i+1}\right \rfloor\right \},
\end{equation*}
it follows immediately  that $|G(P_{(a,b)})|=|G_{gen}|=g(\X_3)$. Hence, as in the proof of Theorem \ref{maingapsgeneric}, we are now left to show that, for each $g\in G$, there exists a function $h_g$ such that $h_g \in L((q-2)D_\infty)$ and $v_P(h_g)=g-1$.

For any $g=jq+k$, let $\cc:=\left\lfloor\frac{k}{3(i+1)}\right\rfloor$. We can then write
\begin{equation*}
    \left\lfloor \frac{k}{3}\right\rfloor = \cc (i+1) + h,
\end{equation*}
where $h$ is an integer such that $0\leq h \leq i$, and
\begin{equation*}
    k = \left\lfloor \frac{k}{3}\right\rfloor\cdot 3 + r = 3\cc (i+1) + 3h + r,
\end{equation*}
where $r$ is an integer such that $0\leq r \leq 2$.
First note that, with this choice of $\cc$, $0\leq \left\lfloor \frac{k}{3}\right\rfloor-(\cc(i+1)+1)\leq i-1$ for all $k$ such that $\left\lfloor \frac{k}{3}\right\rfloor \neq \cc (i+1)$. Indeed,
\begin{equation*}
    \left\lfloor \frac{k}{3}\right\rfloor-(\cc(i+1)+1) \leq i - 1 \quad \Longleftrightarrow \quad \left\lfloor \frac{k}{3}\right\rfloor  -\cc \leq i + \cc i,
\end{equation*}
hence, as $\left\lfloor \frac{k}{3}\right\rfloor = \cc (i+1) + h$, with $h$ an integer such that $0\leq h\leq i$, we obtain
\begin{equation*}
    \left\lfloor \frac{k}{3}\right\rfloor  -\cc \leq i + \cc i \quad \Longleftrightarrow \quad   \cc (i+1) + h -\cc \leq i + \cc i \quad \Longleftrightarrow \quad   h\leq i,
\end{equation*}
which is satisfied.

We now distinguish the following cases.
\begin{enumerate}
    \item If $\left\lfloor \frac{k}{3}\right\rfloor \neq \cc (i+1)$, then we define:
\begin{equation*}
        h_{g}:=\begin{cases}
            F_{P_{(a,b)}}^j\cdot f_{\left\lfloor \frac{k}{3}\right\rfloor - (\cc(i+1)+1)}\cdot f_i^\cc \quad &\mbox{if} \ k \equiv 0 \pmod {3},\\[15pt]
            F_{P_{(a,b)}}^j\cdot (y-b) \cdot f_{\left\lfloor \frac{k}{3}\right\rfloor - (\cc(i+1)+1)}\cdot f_i^\cc \quad &\mbox{if} \ k \equiv 1 \pmod {3},\\[15pt]
            F_{P_{(a,b)}}^j\cdot t_{P_{(a,b)}} \cdot f_{\left\lfloor \frac{k}{3}\right\rfloor - (\cc(i+1)+1)}\cdot f_i^\cc \quad &\mbox{if} \ k \equiv 2 \pmod {3}.\\
        \end{cases}
    \end{equation*}
    \item If $\left\lfloor \frac{k}{3}\right\rfloor = \cc (i+1)$, we define instead:
\begin{equation*}
        h_{g}:=\begin{cases}
            F_{P_{(a,b)}}^j\cdot (y-b) \cdot t_{P_{(a,b)}} \cdot f_{i-1}\cdot f_i^{\cc-1} \quad &\mbox{if} \ k \equiv 0 \pmod {3} \ \mbox{and} \ j\leq m-2-i,\\[15pt]
            F_{P_{(a,b)}}^j \cdot f_{\frac{k}{3}-1} \ &\mbox{if} \ k \equiv 0 \pmod 3 \ \mbox{and} \ j\geq m-1-i,\\[15pt]
            F_{P_{(a,b)}}^j \cdot f_i^\cc \quad &\mbox{if} \ k \equiv 1 \pmod {3},\\[15pt]
            F_{P_{(a,b)}}^j\cdot (y-b) \cdot f_i^\cc \quad &\mbox{if} \ k \equiv 2 \pmod {3}.\\
        \end{cases}
    \end{equation*}
\end{enumerate}
Indeed, for $g=jq+k\in G$, we have the following situation.

\noindent
{\bf Case 1:} $\left\lfloor \frac{k}{3}\right\rfloor \neq \cc (i+1)$.

If $k \equiv 0 \pmod {3}$, then
\begin{equation*}
    v_{P_{(a,b)}}(h_{g})= jq + 3\left(\left\lfloor \frac{k}{3}\right\rfloor - (\cc(i+1)+1)\right) + 2 + 3\cc(i+1) = jq + k - 3 + 2 = jq + k - 1
\end{equation*}
and
\begin{equation*}
\begin{split}
    (h_{g})_{\infty} &\leq (3j + 3\left(\left\lfloor \frac{k}{3}\right\rfloor - (\cc(i+1)+1)+1\right) + 3\cc(i+1))D_\infty\\
    &=(3j + k)D_\infty\\
    &\leq (3j + q -3j -2)D_\infty = (q-2)D_\infty,\\
\end{split}
\end{equation*}
where the last inequality above follows from the fact that $k\leq q-3j-2$.

If $k \equiv 1 \pmod {3}$, then
\begin{equation*}
    v_{P_{(a,b)}}(h_{g})= jq + 1 + 3\left(\left\lfloor \frac{k}{3}\right\rfloor - (\cc(i+1)+1)\right) + 2 + 3\cc(i+1) = jq + (k - 1) - 3 + 3 = jq + k - 1
\end{equation*}
and
\begin{equation*}
\begin{split}
    (h_{g})_{\infty} &\leq (3j + 2 + 3\left(\left\lfloor \frac{k}{3}\right\rfloor - (\cc(i+1)+1)+1\right) + 3\cc(i+1))D_\infty\\
    &=(3j + k + 1)D_\infty\\
    &\leq (3j + q -3j -3)D_\infty = (q-3)D_\infty,\\
\end{split}
\end{equation*}
where the last inequality follows from the fact that $q-3j-2 \equiv 0 \pmod 3$, hence if $k\equiv 1 \pmod 3$, then $k\leq (q-3j-2)-2=q-3j-4$.

If $k \equiv 2 \pmod {3}$, then
\begin{equation*}
    v_{P_{(a,b)}}(h_{g})= jq + 2 + 3\left(\left\lfloor \frac{k}{3}\right\rfloor - (\cc(i+1)+1)\right) + 2 + 3\cc(i+1) = jq + (k - 2) - 3 + 4  = jq + k - 1
\end{equation*}
and
\begin{equation*}
\begin{split}
    (h_{g})_{\infty} &\leq (3j + 3 + 3\left(\left\lfloor \frac{k}{3}\right\rfloor - (\cc(i+1)+1)+1\right) + 3\cc(i+1))D_\infty\\
    &=(3j + k + 1)D_\infty\\
    &\leq (3j + q -3j -2)D_\infty = (q-2)D_\infty,\\
\end{split}
\end{equation*}
where the last inequality follows from the fact that $q-3j-2 \equiv 0 \pmod 3$, hence if $k\equiv 2 \pmod 3$, then $k\leq (q-3j-2)-1=q-3j-3$.

\noindent
{\bf Case 2:} $\left\lfloor \frac{k}{3}\right\rfloor = \cc (i+1)$.

If $k \equiv 0 \pmod {3}$ and $j\leq m-2-i$, then
\begin{equation*}
    v_{P_{(a,b)}}(h_{g})= jq + 1 + 2 + 3(i-1) + 2 + 3(\cc-1)(i+1) = jq + 3\cc(i+1) - 1 = jq + k - 1
\end{equation*}
and
\begin{equation*}
\begin{split}
    (h_{g})_{\infty} &\leq (3j + 2 + 3 + 3i + 3(\cc-1)(i+1))D_\infty\\
    &=(3j + k + 2)D_\infty\\
    &\leq (3j + q -3j -5+2)D_\infty = (q-3)D_\infty,\\
\end{split}
\end{equation*}
since in this case $k\leq q-3j-3 \equiv 2 \pmod 3$ and hence, as $k\equiv 0 \pmod 3$, then $k\leq (q-3j-3)-2=q-3j-5$.

If $k \equiv 0 \pmod {3}$ and $j\geq m-1-i$, then note that, as $3j\geq q-3i-2$, then $k\leq q-3j-2\leq q-(q-3i-2)-2=3i$ and $\frac{k}{3}-1\leq i-1$. Hence, we have that
\begin{equation*}
    v_{P_{(a,b)}}(h_{g})= jq + 3\left(\frac{k}{3}-1\right) + 2 = jq + k - 1
\end{equation*}
and
\begin{equation*}
\begin{split}
    (h_{g})_{\infty} &\leq (3j + 3\left(\frac{k}{3}\right))D_\infty\\
    &=(3j + k )D_\infty\\
    &\leq (3j + q -3j - 2)D_\infty = (q-2)D_\infty,\\
\end{split}
\end{equation*}
since $k\leq q - 3j - 2$ in this case.
\item If $k \equiv 1 \pmod {3}$, then
\begin{equation*}
    v_{P_{(a,b)}}(h_{g})= jq + 3\cc(i+1) = jq + k - 1
\end{equation*}
as $3\left\lfloor \frac{k}{3}\right\rfloor = 3\cc (i+1) = k - 1$.
Moreover,
\begin{equation*}
\begin{split}
    (h_{g})_{\infty} &\leq (3j + 3\cc(i+1))D_\infty\\
    &=(3j + k - 1)D_\infty\\
    &\leq (3j + q -3j - 2 -1)D_\infty = (q-3)D_\infty,\\
\end{split}
\end{equation*}
since $k\leq q - 3j -2$.

If $k \equiv 2 \pmod {3}$, then
\begin{equation*}
    v_{P_{(a,b)}}(h_{g})= jq + 1 + 3\cc(i+1) = jq + (k - 2) + 1 = jq + k - 1
\end{equation*}
as $3\left\lfloor \frac{k}{3}\right\rfloor = 3\cc (i+1) = k - 2$.
Moreover,
\begin{equation*}
\begin{split}
    (h_{g})_{\infty} &\leq (3j + 2 + 3\cc(i+1))D_\infty\\
    &=(3j + k)D_\infty\\
    &\leq (3j + q -3j - 2)D_\infty = (q-2)D_\infty,\\
\end{split}
\end{equation*}
since $k\leq q - 3j -2$.
\end{proof}

\subsection{Final remarks on the Weierstrass points of $\X_3$}
\label{subsec:remarks}

We finish this section by collecting a few further facts on the Weierstrass points of $\X_3$.

\begin{prop}
Only for $q \in \{2,5,8\}$ are all Weierstrass points of $\X_3$ defined over $\F_{q^2}.$
\end{prop}

\begin{proof}
Lemma \ref{lem:Porder_rational} and Theorem \ref{thm:non:rational:weierstrass} imply that a non-rational Weierstrass point exists precisely if there exists $i$ such that $1 \le i \le m-2$, $\gcd(i+1,p)=1$, and $i+1$ does not divide $m$. Since $m$ has at most $m/3+1$ divisors (not counting $m$ itself) and there are at most $\lfloor m/p \rfloor$ multiples of $p$ between $1$ and $m$,  we see that a non-rational Weierstrass point exists if $m-2 > 1+m/3+m/p.$
Since $p \ge 2$, and $m-2 > 1+m/3+m/2$ if and only if $q>53$, this already shows that there exists a non-rational Weierstrass point for all $q>53$. It is trivial to check that $i$ satisfying the conditions exists for $q \in \{11,17,23,29,32,41,47,53\}$, while no such $i$ exists for $i \in \{2,5,8\}.$
\end{proof}

\begin{rem}
It is at this point quite simple to determine the number of distinct possible Weierstrass semigroups $H(P)$ as $P$ varies. Indeed, the possible $\PP$-orders less than or equal to $m-1$ are simply the number of $i$ between $1$ and $m-1$, such that $\gcd(p,i+1)=1$. Counting the semigroup for $P \in {\mathcal O}$ as well, this gives $m-\lfloor m/p \rfloor$ possible semigroups different from the generic semigroup. The generic semigroup corresponds to those points $P$ on $\X_3$ whose $\PP$-order is at least $m$. Hence there are precisely $m-\lfloor m/p \rfloor+1$ possible semigroups.
\end{rem}

\begin{rem}
For $\F_{q^2}$-rational points, we determined the multiplicity and conductors the corresponding Weierstrass semigroups. Using Theorem \ref{maingapsgeneric}, we see that in the generic case, the smallest positive non-gap in $H(P)$ is $q-1$, while the largest gap is $(m-1)q+1$. Hence in the generic case, the multiplicity is $q-1$ and the conductor $(m-1)q+2.$ If $P \in \X_3(\overline{\F}_{q^2}) \setminus \X_3(\F_{q^2})$ has $\PP$-order $i \le m-2$, then Theorem \ref{thm:non:rational:weierstrass} implies quite easily that the largest gap still is $(m-1)q+1$ and therefore that the conductor is $(m-1)q+2.$
\end{rem}

The situation for the multiplicity is more complicated. We show what is going on in the following theorem.

\begin{thm}
Let $P_{(a,b)} \in \X_3(\overline{\F}_{q^2}) \setminus \X_3(\F_{q^2})$. Then the multiplicity of the semigroup $H(P_{(a,b)})$ is $q-2$ or $q-1$. Moreover, the following are equivalent:
\begin{enumerate}
\item The multiplicity of $H(P_{(a,b)})$ is $q-2$.
\item The $\PP$-order $i$ of $\alpha(P_{(a,b)})$ is such that $i+1$ divides $m-1$.
\item $\PP_{m-1}(\alpha(P_{(a,b)}))=0$.
\item The Frobenius of $(a,b)$, that is $\Phi(a,b):=(a^{q^2},b^{q^2})$, lies on the tangent line of the plane curve $y^{q+1}+x^{2m}+x^m=0$ at $(a,b)$.
\end{enumerate}
\end{thm}

\begin{proof}
Comparing the gap set in the generic case and the case described in Theorem \ref{thm:non:rational:weierstrass}, we see that the only difference is that the value of certain gaps is increased by one. Since in the generic case, $1,\dots,q-2$ are gaps and $q-1$ is not a gap, this means that the multiplicity of $H(P_{(a,b)})$ for any $P_{(a,b)} \in \X_3(\overline{\F}_{q^2}) \setminus \X_3(\F_{q^2})$ can be either $q-2$ or $q-1$. Now we show equivalence of the four listed items. For convenience, we write $P=P_{(a,b)}$ and $\alpha=\al(P_{(a,b)})$.

$(1) \Rightarrow (2):$ Assume that $q-2 \in H(P)$ and let $i$ be the $\PP$-order of $\alpha$. Then according to Theorem \ref{thm:non:rational:weierstrass} $q-2$ can be written in the form $(m-2-i-\ell(i+1))q + (\ell+1)(3i+3)$ for some $\ell$ between $0$ and $\lfloor (m-2-i)/(i+1) \rfloor.$ Then necessarily $m-2-i-\ell(i+1)=0$, which is only possible if $\ell =(m-2-i)/(i+1)$ is an integer. Hence $i+1$ divides $m-1$.

$(2) \Rightarrow (3):$ From the definition of the polynomial $\PP_{i+1}(s)$, we see that $((\alpha+\zeta_3)/(\alpha+\zeta_3^2))^{i+1}=1$. If $i+1$ divides $m-1$, this implies that $((\alpha+\zeta_3)/(\alpha+\zeta_3^2))^{m-1}=1$, which in turn implies that $\PP_{m-1}(\alpha)=0.$

$(3) \Rightarrow (4):$ The tangent line $\ell_P$ of the plane curve $y^{q+1}+x^{2m}+x^m=0$ at $(a,b)$ is given by the equation $a^{m-1}(2a^m+1)(x-a)+3b^q(y-b)=0$. Hence $\Phi(a,b)$ lies on $\ell_P$ if and only if $a^{m}(2a^m+1)(a^{q^2-1}-1)+3b^{q+1}(b^{q^2-1}-1)=0$. Using that $b^{q+1}=-a^{2m}-a^m$, we can express all quantities in this equation in terms of $a^m$ and obtain the equivalent equation $a^m((a^m)^{q-1}-1)^2(2(a^m)^q+(a^m)^{q-1}+a^m+2)=0.$ Since $P \not \in \X_3(\F_{q^2})$, we know $a^m \not\in \F_q$ and hence we conclude that
\begin{equation*}
(a^{q^2},b^{q^2}) \in \ell_P  \Leftrightarrow 2(a^m)^q+(a^m)^{q-1}+a^m+2=0.
\end{equation*}
Using that $a^m=\al/(1-\al)$, we conclude that
\begin{equation}\label{eq:phiPontangent}
(a^{q^2},b^{q^2}) \in \ell_P  \Leftrightarrow \alpha^{q-1}+(\alpha-1)^{q-1}+1=0.
\end{equation}
Now let us investigate our assumption: $\PP_{m-1}(\alpha)=0$. This implies \[\left( \frac{\al+\zeta_3}{\al+\zeta_3^2} \right)^{q-2}=1 \text{ and hence } \left( \frac{\al+\zeta_3}{\al+\zeta_3^2} \right)^{q}=\left( \frac{\al+\zeta_3}{\al+\zeta_3^2} \right)^{2},\] which in turn gives
$$0=(\al+\zeta_3)^q (\al+\zeta_3^2)^2-(\al+\zeta_3^2)^q(\al+\zeta_3)^2=(\al^q+\zeta_3^2) (\al+\zeta_3^2)^2-(\al^q+\zeta_3)(\al+\zeta_3)^2.$$
Multiplying everything out and dividing by $\zeta_3^2-\zeta_3$, we find that
\[0=2\alpha^{q+1}-\alpha^q+\alpha^2-2\alpha=\alpha(\alpha-1)(\alpha^{q-1}+(\alpha-1)^{q-1}+1).\]
In light of equation \eqref{eq:phiPontangent}, we obtain that $\Phi(a,b) \in \ell_P$.

$(4) \Rightarrow (1):$ If $\Phi(a,b) \in \ell_P$, then the function $t_{P}/F_P$, see equations \eqref{eq:tangent} and \eqref{eq:F}, has a pole of order $q-2$ at $P_{(a,b)}$ and no other poles. Since we already have seen that $H(P)$ has multiplicity $q-1$ or $q-2$, the conclusion is that the multiplicity is $q-2$.
\end{proof}

\begin{rem}
Let us denote by $W_q$ the total number of Weierstrass points. We have seen that
\begin{eqnarray*}
W_q & = & -(q+1)^2+(q+1)+2(q+1)m+(q+1)^2 \left(\sum_{i=0}^{m-1}\varphi(i+1)-\sum_{i=0}^{(m-1)/p-1} \varphi(p\cdot (i+1))\right)\\
 & = & -(q+1)^2+(q+1)+2(q+1)m+(q+1)^2 \left(\sum_{i=1}^{m}\varphi(i)-\sum_{i=1}^{(m-1)/p} \varphi(p\cdot i)\right).
\end{eqnarray*}
Here the notation $\sum_{i=0}^\xi$ for $\xi \in \mathbb{R}_{\ge 0}$ is shorthand for $\sum_{i=0}^{\lfloor \xi \rfloor}$.

Using iteratively that $$\sum_{i=1}^{(m-1)/p} \varphi(p \cdot i)=(p-1)\sum_{i=1}^{(m-1)/p} \varphi(i)+\sum_{i=1}^{(m-1)/p^2} \varphi(p \cdot i),$$ one obtains that
$$\sum_{i=1}^{(m-1)/p} \varphi(p \cdot i)=\sum_{e=1}^{\lfloor \log_p(m-1) \rfloor}(p-1)\sum_{i=1}^{(m-1)/p^e}\varphi(i).$$
It is well known, see for example \cite[Thm.330]{HW}, that $\sum_{i=1}^{N} \varphi(i) = \frac{3}{\pi^2}N^2 + O(N \log(N))$ asymptotically as $N \to \infty$.

Hence, we see that
\begin{eqnarray*}
\sum_{i=1}^{(m-1)/p} \varphi(p \cdot i) & = & \sum_{e=1}^{\lfloor \log_p(m-1) \rfloor}\frac{3(p-1)(m-1)^2}{\pi^2 p^{2e}}+ O\left(\sum_{e=1}^{\lfloor \log_p(m-1) \rfloor} \frac{m-1}{p^e} \log_p \left( \frac{m-1}{p^e}\right) \right)\\
& = &  \frac{3(p-1)(m-1)^2}{\pi^2}\frac{1-p^2/p^{2 \lfloor \log_p(m-1) \rfloor}}{p^2-1} \\ 
& & \qquad \qquad \qquad \qquad \qquad \qquad +O\left(\int_0^{\log_p(m-1)} \frac{m-1}{p^e} \log_p \left( \frac{m-1}{p^e}\right) de\right)\\
& = & \frac{3(m-1)^2}{\pi^2 (p+1)}- \frac{3p^2}{\pi^2(p+1)}\left(\frac{m-1}{p^{\lfloor \log_p(m-1) \rfloor}}\right)^2+O(q\log(q))\\
& = & \frac{3(m-1)^2}{\pi^2 (p+1)}+O(q\log(q)).
\end{eqnarray*}
Going back to the number of Weierstrass points, we see that
\begin{eqnarray*}
W_q &= & (q+1)^2 \left(\sum_{i=1}^{m-1}\varphi(i)-\sum_{i=1}^{(m-1)/p} \varphi(p\cdot i)\right)+O(q^2)\\
 & = & (q+1)^2\left(\frac{3(m-1)^2}{\pi^2}-\frac{3(m-1)^2}{\pi^2 (p+1)}\right)+O(q^3 \log(q))\\
 & = & \frac{3(m-1)^2(q+1)^2}{\pi^2}\frac{p}{p+1}+O(q^{3}\log(q))\\
 & = & \frac{q^4}{3\pi^2}\frac{p}{p+1}+O(q^{3}\log(q)).
 \end{eqnarray*}
Since the number of rational points is $O(q^3)$, this shows that for large $q$, the number of nonrational Weierstrass points, vastly outnumbers the number of rational Weierstrass points.
\end{rem}

\section{The full automorphism group \texorpdfstring{$\mathrm{Aut}(\mathcal{X}_3)$}{Aut(X3)} of \texorpdfstring{$\mathcal{X}_3$}{X3}}
\label{sec:automorphisms}

It turns out that knowing the Weierstrass semigroup of all $\F_{q^2}$-rational points on $\X_3$, allows us to determine the full automorphism group $\mathrm{Aut}(\mathcal{X}_3)$ of $\X_3$. We devote this section to this. As before $q \equiv 2 \pmod{3}$ and we denote by $p$ the characteristic of $\F_{q^2}$. As discussed in Section \ref{sec:preliminaries}, the function field $\F_{q^2}(\X_3)$ can be seen as a subfield of the Hermitian function field $\F_{q^2}(\mathcal{H})$, and the function field extension $\F_{q^2}(\mathcal{H})/\F_{q^2}(\X_3)$ is an unramified Galois extension of degree $3$ (see Remark \ref{rem:unramified}), with Galois group generated by the automorphism $\tau$, defined in equation~\eqref{eq:tau}. This observation is useful when constructing automorphisms of the curve $\mathcal{X}_3$ (equivalently, of the function field $\F_{q^2}(\X_3)$).

Indeed, a way to find automorphisms of $\F_{q^2}(\X_3)$ is to consider the normalizer $N(\langle\tau\rangle)$ of $\langle\tau\rangle$ in $\mathrm{Aut}(\mathcal{H}) \cong \PGU(3,q)$. Doing so, the group $N(\langle\tau\rangle)/\langle \tau \rangle$ is theoretically guaranteed to be a subgroups of the full automorphism group of the fixed field $\F_{q^2}(\X_3)$ of $\langle\tau\rangle$.
The group $N(\langle\tau\rangle)$ in $\PGU(3,q)$ is a well-known maximal subgroup stabilizing a self-polar triangle, see \cite{HKT}*{Theorem A.10}. It has order $6(q+1)^2/\gcd(3,q+1)=2(q+1)^2$ and is isomorphic to the semidirect product of an abelian group of order $(q+1)^2/3$ containing $\tau$ and a symmetric group of order $6$. This explains the structure of the automorphism group described in Lemma \ref{lem:someauto}.

We now begin our study of the full automorphism group of $\X_3$. Recall that ${\mathcal O}_0:=\{P_{0}^1,\ldots ,P_{0}^m\}$, ${\mathcal O}_\infty:=\{P_{\infty}^1,\ldots ,P_{\infty}^m\}$ and ${\mathcal O}_m=\{P_{(a,0)} \mid a^m+1=0\}$. Moreover $\mathcal{O}={\mathcal O}_0 \cup {\mathcal O}_\infty \cup {\mathcal O}_m$ as in equation \eqref{eq:O}.

\begin{lem} \label{lem:action}
Let $\mathcal{O}$ be the set defined in equation \eqref{eq:O}. Then $\mathcal{O}$ is an orbit of $\mathrm{Aut}(\mathcal{X}_3)$.
\end{lem}

\begin{proof}
Since $\mathcal{X}_3$ is $\mathbb{F}_{q^2}$-maximal, its full automorphism group $\mathrm{Aut}(\mathcal{X}_3)$ is defined over $\mathbb{F}_{q^2}$ and hence acts on the set $\mathcal{X}_3(\mathbb{F}_{q^2})$, see for example \cite[Lemma 2.4]{BMT}.
Let $H(P_{(a,b)})$ and $H(P)$ be the Weierstrass semigroups at a point $P_{(a,b)} \in \mathcal{X}_3(\mathbb{F}_{q^2}) \setminus \mathcal{O}$ and at $P \in \mathcal{O}$, respectively. Since the semigroups $H(P_{(a,b)})$ and $H(P)$ are not the same (see Theorems \ref{thm:O}, \ref{thm:special:short:orbit} and \ref{thm:rationalcase:other:semigroups}), $\mathrm{Aut}(\mathcal{X}_3)$ acts separately on $\mathcal{O}$ and $\mathcal{X}_3(\mathbb{F}_{q^2}) \setminus \mathcal{O}$. Moreover, since from Corollary \ref{cor:orbitO} $\mathcal{O}$ is an orbit of $G \subseteq \mathrm{Aut}(\mathcal{X}_3)$, we deduce that $\mathcal{O}$ is also an orbit of the entire $\mathrm{Aut}(\mathcal{X}_3)$.
\end{proof}

We now use that $\mathcal O$ is an orbit, to start investigating the $p$-Sylow subgroup of $\mathrm{Aut}(\mathcal{X}_3)$.

\begin{lem} \label{lem:Splessthanq}
Let $q \geq 11$. Let $S_p$ denote a Sylow subgroup of $|\mathrm{Aut}(\mathcal{X}_3)|$. Then $|S_p|<q$.
\end{lem}

\begin{proof}
Since $S_p$ acts on $\mathcal{O}$ by Lemma \ref{lem:action}, we see that $S_p$ has at least one fixed point $P \in \mathcal{O}$. Without loss of generality, we can assume $P=P_{(a,0)}$, for some $a$ such that $a^m+1=0$. Since $\mathcal{X}_3$ has $p$-rank zero, $S_p$ acts with long orbits on $\mathcal{O} \setminus \{P_{(a,0)}\}$, see \cite[Lemma 11.129]{HKT}. This implies that $|S_p| \leq q$.

Now suppose that $|S_p|=q$. Then $\mathrm{Aut}(\mathcal{X}_3)$ acts $2$-transitively on $\mathcal{O}$ and the stabilizer of two points is cyclic in this action, since it is of order relatively prime to $p$ (see \cite[Theorem 11.49]{HKT}). Moreover, from \cite[Theorem 1.1]{KOS}, $\mathrm{Aut}(\mathcal{X}_3)$ has a regular normal subgroup $N$, unless:
\begin{itemize}
    \item $\mathrm{Aut}(\mathcal{X}_3)$ is isomorphic to either $\mathrm{PSL}(2,q)$, $\mathrm{PGL}(2,q)$, or
    \item $q=\bar q^3$ and $\mathrm{Aut}(\mathcal{X}_3)$ is isomorphic to $\mathrm{PSU}(3,\bar q)$ or $\PGU(3,\bar q)$, or
    \item $\mathrm{Aut}(\mathcal{X}_3)$ is isomorphic to the Suzuki group $Sz(\bar{q})$ where $q=\bar{q}^2$.
\end{itemize}
The first two possibilities can be excluded, since in that case $|\mathrm{Aut}(\mathcal{X}_3)|$ would not be divisible by $2(q+1)^2$. Further, if $\mathrm{Aut}(\mathcal{X}_3)$ would be isomorphic to the Suzuki group $Sz(\bar{q})$, then the characteristic is two and $q=\bar{q}^2$ is an even power of two. However, this is impossible, since $q \equiv 2 \pmod{3}$. This means that $\mathrm{Aut}(\mathcal{X}_3)$ has a regular normal subgroup $N$. Then, from \cite{BW}*{Theorem 1.7.6}, we see that $|\mathcal{O}|=q+1=\ell^h$ for some $h\in \mathbb{Z}_{>0}$ and some prime number $\ell$. If $q$ is odd, this cannot happen as $q+1$ is divisible by $6$. If $q$ is even, we would get $|\mathcal{O}|=q+1=2^n+1=\ell^h$. If $h=1$, this would mean that $\ell$ is a Fermat prime, which is only possible if $n$ is a power of two. However, since $n$ is odd, this would imply $n=1$. This is impossible, since $q \ge 11$. If $h >1$, then from Catalan's Conjecture (Mihailescu's theorem \cite{Mih}), we see that the only possibility is that $\ell=3$ and $n=3$. This is again not possible, since we assumed that $q \ge 11$. Hence, we conclude that the only possibility is $|S_2|<q$.
\end{proof}

Next is a lemma that will allow us to identify certain automorphisms of $\X_3$.

\begin{lem} \label{lem:lift}
Let $\alpha \in \mathrm{Aut}(\mathcal{X}_3)$ and suppose that $\alpha(x)$ is a cube, when seen as a function of the Hermitian function field $\overline{\mathbb{F}}_{q^2}(\mathcal{H})$. Then $\alpha$ can be lifted to an automorphism $\bar\alpha$ of $\mathcal{H}$.
\end{lem}

\begin{proof}
Since $\alpha$ is an automorphism of $\mathcal{X}_3$, we know that $\alpha(y)^{q+1}+\alpha(x)^{(q+1)/3}+\alpha(x)^{2(q+1)/3}=0$.
Let $\alpha(x)=w^3$, where $w=w(u,v) \in \overline{\mathbb{F}}_{q^2}(\mathcal{H})$, and define
$$\bar\alpha(u)=w \quad \textrm{and} \quad \bar\alpha(v)=\frac{\alpha(y)}{\bar\alpha(u)}.$$
Then
$$\bar\alpha(u)^{q+1}+\bar\alpha(v)^{q+1}+1=w^{q+1}+\frac{\alpha(y)^{q+1}}{w^{q+1}}+1=\frac{\alpha(x)^{2(q+1)/3}+\alpha(y)^{q+1}+\alpha(x)^{(q+1)/3}}{w^{q+1}}=0.$$
This means that $\bar\alpha$ preserves the defining equation of the Hermitian function field, and defines an automorphism of $\mathcal{H}$. 
\end{proof}

Note that since all automorphisms of $\mathcal H$ are defined over $\F_{q^2}$, the automorphism $\bar{\alpha}$ will also be defined over $\F_{q^2}$. Therefore, if $\alpha(x)$ is a cube in $\overline{\mathbb{F}}_{q^2}(\mathcal{H})$, it was necessarily already a cube in $\mathbb{F}_{q^2}(\mathcal{H}).$

\subsection{The full automorphism group $\mathrm{Aut}(\mathcal{X}_3)$, $q$ odd}

We wish to use the information that $\mathcal{O}$ is an orbit of $\mathrm{Aut}(\mathcal{X}_3)$ to show that, for odd $q$. If $q=5$, the plane curve defined by the (affine) equation $X^5+X=Y^3$, is birationally equivalent to $\X_3$. The corresponding isomorphism of function fields is describes as $x=wX+(wX)^{-1},y=Y/X$, where $w^2=2$. This curve is known to have an automorphism group that is isomorphic to a semidirect product of a cyclic group of order $3$ with $\mathrm{PGL}(2,5)$, see \cite{HKT}*{Theorem 12.11}. In particular $|\mathrm{Aut}(\X_3)|=360$ if $q=5$, which is five times the cardinality of the group of automorphisms $G$ described in Lemma \ref{lem:someauto}.

From now, we assume in this subsection that $q \ge 11$ and $q$ is odd. It turns out that in this case, the automorphism group of $\mathcal{X}_3$ actually coincides with $G$. To see why, let us first prove under the aforementioned hypothesis on $q$ that $\mathrm{Aut}(\mathcal{X}_3)$ is tame, that is, it does not contain any element of order $p$.

\begin{lem} \label{lem:tame}
Let $q \geq 11$ and $q$ odd. Then $|\mathrm{Aut}(\mathcal{X}_3)|$ is not divisible by the characteristic $p$ of the field $\mathbb{F}_{q^2}$.
\end{lem}

\begin{proof}
Suppose by contradiction that $\mathrm{Aut}(\mathcal{X}_3)$ admits a Sylow $p$-subgroup $S_p$ of order $p^i$ for some $i \geq 1$. As we have seen in the proof of Lemma \ref{lem:Splessthanq}, we may assume that $S_p$ fixes $P \in \mathcal{O}$, where $P=P_{(a,0)}$, for some $a$ such that $a^m+1=0$ and that $S_p$ acts with long orbits on $\mathcal{O} \setminus \{P_{(a,0)}\}$. Further by Lemma \ref{lem:Splessthanq}, we may assume that $|S_p|<q$.

Recall that the automorphism $\sigma: (x,y) \mapsto (x,\delta y)$, where $\delta$ is a primitive $(q+1)$-th root of unity, fixes the set ${\mathcal O}_m$ point-wise, while it acts transitively on the sets ${\mathcal O}_0$ and ${\mathcal O}_\infty$. From this, it follows that $\sigma$ normalizes $S_p$ (see \cite[Theorem 11.49]{HKT}) and preserves the orbit of $S_p$ containing ${\mathcal O}_m$. We have thus two possibilities for a fixed $\bar{a}$: either the orbit of $S_p$ containing $P_{(\bar{a},0)}$ is contained in ${\mathcal O}_m$, or it contains entirely either ${\mathcal O}_0$ or ${\mathcal O}_\infty$. In the second case, we would get that $|S_p| \geq (q+1)/3+1$ and hence $|S_p|=q$, which is not possible. Therefore, we can deduce that, for all $\bar{a}$ with $\bar{a}^m+1=0$, the $S_p$-orbit of $P_{(\bar{a},0)}$ is contained in ${\mathcal O}_m$. Since $S_p$ acts on $\mathcal{O}={\mathcal O}_m \cup {\mathcal O}_0 \cup {\mathcal O}_\infty$, $S_p$ must then act with long orbits on ${\mathcal O}_0 \cup {\mathcal O}_\infty$, which is a set of cardinality $2(q+1)/3$. We hence obtain the desired contradiction, as $2(q+1)/3$ is not divisible by $p$.
\end{proof}

\begin{thm}\label{thm:fullaut:qodd}
Let $q \geq 11$, $q$ odd. Then $\mathrm{Aut}(\mathcal{X}_3)=G$.
\end{thm}

\begin{proof}
Suppose by contradiction that $|\mathrm{Aut}(\mathcal{X}_3)|>|G|$. Let $G_{P_{(a,0)}}$ be the stabilizer in G of $P_{(a,0)}$, for an $a$ such that $a^m+1=0$. Since, by the orbit-stabilizer theorem, $|G|=|\mathcal{O}||G_{P_{(a,0)}}|$ and, by Lemma \ref{lem:action}, $\mathcal{O}$ is an orbit of $\mathrm{Aut}(\mathcal{X}_3)$, the stabilizer $\mathrm{Aut}(\mathcal{X}_3)_{P_{(a,0)}}$ of $P_{(a,0)}$ in $\mathrm{Aut}(\X_3)$ contains some extra automorphism $\gamma \not\in G_{P_{(a,0)}}$.
Let $C_{q+1}$ be the cyclic group generated by $\sigma: (x,y) \mapsto (x,\delta y)$, where $\delta$ is a primitive $(q+1)$-th root of unity. Then, since $\mathrm{Aut}(\mathcal{X}_3)_{P_{(a,0)}}$ is cyclic (as follows from the fact that $\mathrm{Aut}(\mathcal{X}_3)$ is of order relatively prime to $p$), $\gamma$ commutes with $C_{q+1}$ and hence it acts on its fixed points (and, in general, orbits). This means that $\gamma$ acts on the sets ${\mathcal O}_m$ and ${\mathcal O}_0 \cup {\mathcal O}_\infty$, because the set ${\mathcal O}_m$ is exactly the set of fixed points of $C_{q+1}$. Since ${\mathcal O}_0$ and ${\mathcal O}_\infty$ are orbits of $C_{q+1}$ of the same length, either $\gamma$ fixes both ${\mathcal O}_0$ and ${\mathcal O}_\infty$, or interchanges them.

If $\gamma$ fixes both ${\mathcal O}_0$ and ${\mathcal O}_\infty$, then it fixes the divisor of $x$ from equation \eqref{eq:divs}. This means that $\gamma(x)=\lambda x$, for some constant $\lambda$. Hence, $\gamma(x)$ is a cube in $\overline{\mathbb{F}}_{q^2}(\mathcal{H})$, as $x=u^3$ and $\lambda$ is a constant.
Suppose instead that $\gamma$ interchanges ${\mathcal O}_0$ and ${\mathcal O}_\infty$. Then, $\gamma$ maps the divisor of $x$ to the divisor of $\frac{1}{x}$, meaning that there exists a constant $\lambda$ such that $\gamma(x)=\frac{\lambda}{x}$. Hence, in all cases $\alpha(x)$ is a cube in $\overline{\mathbb{F}}_{q^2}(\mathcal{H})$.

From Lemma \ref{lem:lift}, $\gamma$ can be lifted to an automorphism $\bar\gamma$ of the Hermitian curve $\mathcal{H}$ acting on the set of $3(q+1)$ points above those in $\mathcal{O}$. Those points are geometrically the intersection of the Hermitian curve $\mathcal{H}$ with $3$ lines intersecting each other in $3$ points outside $\mathcal{H}$, that is a self-polar triangle. Since this shows that $\gamma$ is induced by $N(\langle\sigma\rangle)$, then $\gamma \in G$, which gives a contradiction.
\end{proof}

\subsection{The full automorphism group $\mathrm{Aut}(\mathcal{X}_3)$, $q$ even}

We now turn our attention to the case where $q$ is even, that is to say when $q=2^n$, $n$ odd.
If $q=2$, the curve is isomorphic to the Hermitian curve over $\F_4$ and therefore has $\PGU(3,2)$ as automorphism group, which contains $216$ elements. Here only automorphisms defined over $\F_{q^2}$ were considered. Hence in this case, there are twelve times more automorphisms than described in Lemma \ref{lem:someauto}.
If $q=8$, the automorphism group of $\mathcal{X}_3$ is known, as in this case $\X_3$ is isomorphic to the Giulietti-Korchm\'aros maximal curve (see \cite{GK}). This curve can for example be given as a plane curve with affine equation $Y^9=(X^2+X)(X^2+X+1)^{3}$. An explicit isomorphism on the level of function fields is then given by $X=\zeta_3+(x^5+x^4+x^3)/y^9$ and $Y=(x^5+x^4+x^3)/y^8.$ Hence for $q=8$, the automorphism group of $X_3$ is a semidirect product of a cyclic group of order $3$ and $\PGU(3,2)$, resulting in $648$ automorphisms, four times more than the group from Lemma \ref{lem:someauto} contains.

From the remainder of this subsection, we will assume that $q=2^n$, $n$ odd and at least five. We will now show that in this case the automorphism group of $\mathcal{X}_3$ coincides with the group $G$ from Lemma \ref{lem:someauto}. To this aim, a similar argument as in the previous subsection will be provided. Of course in this case we cannot prove that $\mathrm{Aut}(\mathcal{X}_3)$ is tame, as $G$ itself is non-tame. We will in fact first prove that, if a Sylow $2$-subgroup of $\mathrm{Aut}(\mathcal{X}_3)$ has order larger than $2$, then its cardinality must be $q/2$.

\begin{lem} \label{lem:ordersylow}
Let $n \geq 5$ and $q=2^n$. Let also $S_2$ denote a Sylow $2$-subgroup of  $\mathrm{Aut}(\mathcal{X}_3)$. Then either $|S_2|=2$ or $|S_2|=q/2$. In the latter case, a $2$-Sylow $S_2$ fixing a point $P_{(a,0)}$, with $a^m+1=0$, acts on $\mathcal{O}$ with the following 3 orbits:
\begin{itemize}
    \item $\{P_{(a,0)}\}$,
    \item ${\mathcal O}_1^{S_2}:={\mathcal O}_0 \cup \{P_{(\beta_1,0)}, \ldots, P_{(\beta_{(q-2)/6,0})}\}$,
    \item ${\mathcal O}_2^{S_2}:={\mathcal O}_\infty \cup \{ P_{(\gamma_1,0)}, \ldots, P_{(\gamma_{(q-2)/6},0)}\}$,
\end{itemize}
where $\{P_{(a,0)}\}$, $\{P_{(\beta_1,0)}, \ldots, P_{(\beta_{(q-2)/6,0})}\}$ and $\{P_{(\gamma_1,0)}, \ldots, P_{(\gamma_{(q-2)/6,0})}\}$ is a suitably chose partition of ${\mathcal O}_m.$
\end{lem}

\begin{proof}
Let $S_2$ be of order $2^i$ for some $i \geq 1$. Just as in the proof of Lemma \ref{lem:tame}, we may assume that $S_2$ fixes $P \in \mathcal{O}$, where $P=P_{(a,0)}$, for some $a$ such that $a^m+1=0$ and that $S_2$ acts with long orbits on $\mathcal{O} \setminus \{P_{(a,0)}\}$. Further by Lemma \ref{lem:Splessthanq}, we may assume that $|S_2|<q$.

Recall that the automorphism $\sigma: (x,y) \mapsto (x,\delta y)$, where $\delta$ is a primitive $(q+1)$-th root of unity, fixes $P_{(a,0)}$ and hence normalizes $S_2$, from \cite[Theorem 11.49]{HKT}. Moreover, $\sigma$ fixes the set ${\mathcal O}_m$ point-wise, while it acts transitively on ${\mathcal O}_0$ and ${\mathcal O}_\infty$. This means that $\sigma$ preserves the orbit of $S_2$ containing $P_{(a,0)}$, for $a$ such that $a^m+1=0$. We have thus two possibilities for a fixed $a$: either the orbit of $S_2$ containing $P_{(a,0)}$ is contained in ${\mathcal O}_m$, or it contains entirely either ${\mathcal O}_0$ or ${\mathcal O}_\infty$.

If the second case never occurs, then $S_2$ acts semiregularly on ${\mathcal O}_0 \cup {\mathcal O}_\infty$, which is a set of cardinality $2(q+1)/3$. This implies that $|S_2|=2$. If the second case occurs for some $a$, then we get that $|S_2| \geq (q+1)/3+1$ and hence $|S_2|=q/2$. Note that in this case the only possible configuration of orbits of $S_2$ acting on the $q$ points in ${\mathcal O}_m \setminus \{P_{(a,0)}\}$ is that $S_2$ has exactly $2$ orbits of length $q/2$: one ${\mathcal O}_1^{S_2}$ containing ${\mathcal O}_0$ and $(q-2)/6$ points of ${\mathcal O}_m$, and another one ${\mathcal O}_2^{S_2}$ containing ${\mathcal O}_\infty$ and the remaining $(q-2)/6$ points in ${\mathcal O}_m$.
\end{proof}

We now exclude the second case in Lemma \ref{lem:ordersylow}.

\begin{lem}\label{lem:S2notq2}
The case $|S_2|=q/2$ cannot occur.
\end{lem}

\begin{proof}
Suppose by contradiction $|S_2|=q/2$. With notation as in Lemma \ref{lem:ordersylow}, we can assume that $S_2$ acts on $\mathcal{O}$ with three orbits $\{P_{(a,0)}\}$, ${\mathcal O}_1^{S_2}$ and ${\mathcal O}_2^{S_2}\}$.  The cyclic group $C_{q+1}$, generated by $\sigma: (x,y) \mapsto (x,\delta y)$, where $\delta$ is a primitive $(q+1)$-th root of unity, fixes any point in ${\mathcal O}_m$, in particular $P_{(a,0)}$, and hence normalizes $S_2$. In particular, the group $L$ generated by $\sigma$ and the elements of $S_2$ has $|S_2| (q+1)$ many elements. Since the stabilizer of two points is tame and cyclic, we conclude that $C_{q+1}$ is the two points stabilizer of the points $P_{(a,0)}$ and any other $P \in {\mathcal O}_m.$

Using the notation from Lemma \ref{lem:ordersylow}, choose $\gamma \in S_2$ be such that $\gamma(P_{(\beta_1,0)})=P_{(\beta_2,0)}$, with $P_{(\beta_1,0)},P_{(\beta_2,0)} \in {\mathcal O}_1^{S_2}$ distinct. Such a $\gamma$ exists, since $P_{(\beta_1,0)}$ and $P_{(\beta_2,0)}$ are in the same orbit under the action of $S_2$. Then $\gamma^{-1} \cdot \sigma \cdot \gamma$ fixes $P_{(a,0)}$ and
$\gamma^{-1} \cdot \sigma \cdot \gamma(P_{(\beta_1,0)})=\gamma^{-1} \cdot \sigma(P_{(\beta_2,0)})=\gamma^{-1}(P_{(\beta_2,0)})=P_{(\beta_1,0)}$. Hence, $\gamma^{-1} \cdot \sigma \cdot \gamma$ is an element of order $q+1$ fixing both $P_{(a,0)}$ and $P_{(\beta_1,0)}$. Hence $\gamma^{-1} \cdot \sigma \cdot \gamma \in C_{q+1}$ and more specifically $\gamma^{-1} \cdot \sigma \cdot \gamma=\sigma^k$, where $(k,q+1)=1$. Moreover, since $C_{q+1}$ normalizes $S_2$, there exists $\tilde{\gamma} \in S_2$ such that $\sigma \cdot \gamma=\tilde{\gamma} \cdot \sigma$. Hence $\mathrm{id}=\gamma^{-1} \cdot \sigma \cdot \gamma \cdot \sigma^{-k} = \gamma^{-1}\tilde{\gamma}  \cdot  \sigma^{1-k}$. Since $S_2 \cap C_{q+1}=\{\mathrm{id}\}$, this implies that $k=1$ and hence that $\gamma$ and $\sigma$ commute.

Now let $\iota$ be a suitable power of $\gamma$ such that $\iota$ has order two. Then for any $P_{(\beta_j,0)} \in {\mathcal O}_1^{S_2}$, we have $\iota(P_{(\beta_j,0)}) \in {\mathcal O}_1^{S_2}$, since $S_2$ acts on ${\mathcal O}_1^{S_2}$. On the other hand, using that $\sigma$ and $\iota$ commute and that $\sigma$ fixes all points in ${\mathcal O}_m$, we have $\sigma \cdot \iota (P_{(\beta_j,0)})=\iota \cdot \sigma (P_{(\beta_j,0)})=\iota(P_{(\beta_j,0)}).$ Hence $\iota(P_{(\beta_j,0)})$ is a fixed point of $\sigma$, which implies that $\iota(P_{(\beta_j,0)}) \in {\mathcal O}_m$. We conclude that $\iota(P_{(\beta_j,0)}) \in {\mathcal O}_1^{S_2} \cap {\mathcal O}_m=\{P_{(\beta_1,0)},\dots,P_{(\beta_{(q-2)/6},0)}\}$. In other words: $\iota$ acts on $\{P_{(\beta_1,0)},\dots,P_{(\beta_{(q-2)/6},0)}\}$. Since $(q-2)/6=(q/2+1)/3$ is an odd number, this implies that $\iota$ has, apart from $P_{(a,0)}$, at least one more fixed point. However, since the characteristic is two, this is impossible according to \cite[Lemma 11.129]{HKT}.
\end{proof}

We are now ready to compute $\mathrm{Aut}(\mathcal{X}_3)$ when $q$ is even.

\begin{thm}\label{thm:fullaut:qeven}
Let $q=2^n$, $n \geq 5$ odd. Then $\mathrm{Aut}(\mathcal{X}_3)=G$.
\end{thm}

\begin{proof}
Combining Lemmas \ref{lem:ordersylow} and \ref{lem:S2notq2}, we conclude that $|S_2|=2$. Suppose by contradiction that $|\mathrm{Aut}(\mathcal{X}_3)|>|G|$. Let $G_{P_{(a,0)}}$ be the stabilizer in G of $P_{(a,0)}$, for $a$ such that $a^m+1=0$. Since, by the orbit-stabilizer theorem, $|G|=|\mathcal{O}||G_{P_{(a,0)}}|$ and, by Lemma \ref{lem:action}, $\mathcal{O}$ is an orbit of $\mathrm{Aut}(\mathcal{X}_3)$, the stabilizer $\mathrm{Aut}(\mathcal{X}_3)_{P_{(a,0)}}$ of $P_{(a,0)}$ in $\mathrm{Aut}(\X_3)$ contains some extra automorphism $\gamma \not\in G_{P_{(a,0)}}$. Also, since $|S_2|=2$ and $|G|=2(q+1)^2$, $\gamma$ can be assumed to be of odd order.

Let $C_{q+1}$ be the cyclic group generated by $\sigma: (x,y) \mapsto (x,\delta y)$, where $\delta$ is a primitive $(q+1)$-th root of unity. Then, since the tame part of $\mathrm{Aut}(\mathcal{X}_3)_{P_{(a,0)}}$ is cyclic, $\gamma$ commutes with $C_{q+1}$ and hence it acts on its fixed points (and, in general, orbits). At this point, the remainder of the proof is exactly the same as the proof of Theorem \ref{thm:fullaut:qodd}.
\end{proof}



\end{document}